\documentclass{amsart}
\usepackage{amssymb}
\usepackage[mathcal]{euscript}
\usepackage[cmtip,all]{xy}
\usepackage{amsmath}\usepackage{amsfonts}
\usepackage{verbatim} 
\usepackage{hyperref} 
\usepackage{graphicx} 
\usepackage{verbatim} 
\usepackage{mathtools} 
\usepackage{dsfont} 
\usepackage{tabu} 

\newcommand{\bi}{\textnormal{\textbf{i}}}  \newcommand{\cB}{\mathcal{B}} \newcommand{\cC}{\mathcal{C}} \newcommand{\FF}{\mathbb{F}} \newcommand{\ZZ}{\mathbb{Z}}
 \newcommand{\NN}{\mathbb{N}} 
\newcommand{\cT}{\mathcal{T}} \newcommand{\cO}{\mathcal{O}} \newcommand{\cV}{\mathcal{V}} \newcommand{\cW}{\mathcal{W}} 
\newcommand{\med}{\;|\;}  
\DeclareMathOperator{\supp}{\mathrm{Supp}\,} \DeclareMathOperator{\Aut}{\mathrm{Aut}} \DeclareMathOperator{\Diag}{\mathrm{Diag}} \DeclareMathOperator{\Stab}{\mathrm{Stab}}   
\DeclareMathOperator{\End}{\mathrm{End}}  \DeclareMathOperator{\Sym}{\mathrm{Sym}}  
\DeclareMathOperator{\AAut}{\mathbf{Aut}} \DeclareMathOperator{\OO}{\mathbf{O}} 
 
\DeclareMathOperator{\chr}{\mathrm{char}\,}  \DeclareMathOperator{\lspan}{\mathrm{span}}  \newcommand{\CD}{\mathfrak{CD}}

\newcommand{\Ort}{\mathrm{O}} \newcommand{\Spin}{\mathrm{Spin}} \newcommand{\GL}{\mathrm{GL}} 
 \newcommand{\Univ}{\mathcal{U}} 
\DeclareMathOperator{\Hom}{\mathrm{Hom}\,} \DeclareMathOperator{\id}{\mathrm{id}}

 \DeclareMathOperator{\im}{\textnormal{im}}

\newcommand{\cA}{\mathcal{A}} \newcommand{\cH}{\mathcal{H}} \newcommand{\cS}{\mathcal{S}}
\newcommand{\inv}{{}^-} \newcommand{\kan}{\mathfrak{K}}
\DeclareMathOperator{\ii}{\textbf{i}} \DeclareMathOperator{\charTrivial}{\mathds{1}}

\newtheorem{theorem}{Theorem}
\newtheorem{proposition}[theorem]{Proposition}
\newtheorem{lemma}[theorem]{Lemma}
\newtheorem{corollary}[theorem]{Corollary}

\theoremstyle{definition}
\newtheorem{df}[theorem]{Definition}
\newtheorem{example}[theorem]{Example}
\newtheorem{notation}[theorem]{Notation}

\theoremstyle{remark}
\newtheorem{remark}[theorem]{Remark}

\numberwithin{equation}{section} 
\numberwithin{theorem}{section} 
\begin{document}
\title[Fine gradings on Kantor systems of Hurwitz type]{Fine gradings on Kantor systems \\ of Hurwitz type}

\author[D. Aranda-Orna]{Diego Aranda-Orna}
\address{Departamento de Matem\'aticas y Computaci\'on,
Universidad de La Rioja, 26006, Logro\~no, Spain}
\email{diego.aranda.orna@gmail.com}

\author[A.S. C\'ordova-Mart\'inez]{Alejandra S. C\'ordova-Mart\'inez}
\address{Departamento de Matem\'{a}ticas
 e Instituto Universitario de Matem\'aticas y Aplicaciones,
 Universidad de Zaragoza, 50009 Zaragoza, Spain}
\email{sarina.cordova@gmail.com}

\thanks{Both authors are supported by grant MTM2017-83506-C2-1-P (AEI/FEDER, UE). The first author was affiliated with the Instituto Tecnol\'ogico de Castilla y Le\'on (Burgos, Spain) during most part of this work. The second author acknowledges support by grant S60$\_$20R (Gobierno de Arag\'on, Grupo de investigaci\'on “Investigaci\'on en Educaci\'on Matem\'atica”).}
\date{}

\begin{abstract}
We give a classification up to equivalence of the fine group gradings by abelian groups on the Kantor pairs and triple systems associated with Hurwitz algebras (i.e., unital composition algebras), under the assumption that the base field is algebraically closed of characteristic different from $2$. The universal groups and associated Weyl groups are computed. We also determine, in the case of Kantor pairs, the induced (fine) gradings on the associated Lie algebras given by the Kantor construction.
\end{abstract}

\maketitle


\section{Introduction}

There is a close relation between Lie algebras and Kantor pairs which extends also to gradings. 
From a Kantor pair $\cV=(\cV^-,\cV^+)$ we can construct a 5-graded Lie algebra 
$$\kan(\cV) = \kan(\cV)^{-2} \oplus \kan(\cV)^{-1} \oplus \kan(\cV)^{0} \oplus \kan(\cV)^{1} \oplus \kan(\cV)^{2},$$ 
called \textit{the Kantor Lie algebra of $\cV$}, by means of the Kantor construction (see \cite[\S3--4]{AF99}). Subspaces $\kan(\cV)^{1}$ and $\kan(\cV)^{-1}$ are identified with $\cV^+$ and $\cV^-$, respectively.

Conversely, from a 5-graded Lie algebra 
$$L = L_{-2} \oplus L_{-1} \oplus L_{0} \oplus L_{1} \oplus L_{2},$$
we obtain a Kantor pair given by $\cV=(L_{-1}, L_{1})$ (see \cite[\S4.2]{AFS17} and references therein), where the triple products are defined using the Lie bracket. In these constructions, Jordan pairs correspond to $3$-graded Lie algebras.

The Kantor construction can be used to extend automorphisms and gradings from a Kantor pair $\cV$ to the associated Lie algebra $\kan(\cV)$. The main goal of this work is to classify fine gradings, up to equivalence, on the (simple) Kantor pairs and triple systems associated with Hurwitz algebras. The fine gradings on Kantor pairs of Hurwitz type correspond to the fine gradings on the associated simple Lie algebras which are compatible with the main $\ZZ$-grading associated with the Kantor construction. In order to reach our goal, automorphisms and orbits are studied for these Kantor systems.
 
\bigskip

Throughout this paper, we will always assume that the base field $\FF$ is algebraically closed of characteristic different from $2$, unless otherwise stated.

\smallskip

This paper is structured as follows:

In Section~\ref{section.defs}, we recall some basic definitions and results related to gradings, structurable algebras, Kantor systems (i.e., Kantor pairs and Kantor triple systems), Peirce decompositions, and Hurwitz algebras; a few original results will be proven here too.

In Section~\ref{section.generalities}, some general results related to automorphisms and gradings on Kantor systems are proven.

Section~\ref{section.automorphisms.and.orbits} is aimed to study the automorphisms and orbits of Kantor systems of Hurwitz type. Unexpectedly, an exceptional case occurs for $2$-dimensional Kantor pairs of Hurwitz type if $\chr\FF = 3$. This case, in Section~\ref{section.induced.gradings}, is shown to be related to the Lie algebra $\mathfrak{a}_2$, which is also exceptional if $\chr\FF = 3$.

A classification of the fine gradings up to equivalence, for Kantor systems of Hurwitz type, is given in Section~\ref{section.classification.gradings}. The Weyl groups of these fine gradings are computed in Section~\ref{section.weyl}, and the induced gradings on Lie algebras via the Kantor construction are described in Section~\ref{section.induced.gradings}.

\section{Definitions and preliminaries} \label{section.defs}
\subsection{Gradings on algebras}

Now we will recall the basic definitions of gradings on algebras.

Let $\cA$ be an $\FF$-algebra (not necessarily associative) and $G$ a group. A {\em $G$-grading} on $\cA$ is a vector space decomposition
$$ \Gamma:\;\cA=\bigoplus_{g\in G} \cA_g $$
such that $\cA_g \cA_h\subseteq \cA_{gh}$ for all $g,h\in G$. Given a $G$-grading on $\cA$, we will also say that $\cA$ is a {\em $G$-graded algebra}. The nonzero elements $x\in\cA_g$ are said to be {\em homogeneous of degree $g$}, and we have a {\em degree map} $\deg_\Gamma (x)=g$ (also denoted by $\deg$ if there is no ambiguity with other gradings). The subspace $\cA_g$ is called {\em homogeneous component of degree $g$}. The set $\supp \Gamma := \{g\in G \med \cA_g\neq 0\}$ is called the {\em support} of the grading. 

Consider a group homomorphism $\alpha: G\rightarrow H$. Let $\Gamma: \cA = \bigoplus_{g\in G} \cA_g$ be a $G$-grading on  $\cA$. Then ${}^{\alpha}\Gamma: \cA = \bigoplus_{h\in H} \cA'_h$, where $\cA'_h=\bigoplus_{g\in \alpha^{-1}(h)} \cA_g$, is an $H$-grading on $\cA$ called the \textit{induced grading} from $\Gamma$ by the homomorphism $\alpha$.

Recall that an {\em involution} is an $\FF$-linear antiautomorphism of order $2$. For the case of an algebra with involution $(\cA, \inv)$, we will also require the homogeneous components to be invariant by the involution, that is, $\overline{\cA_g} = \cA_g$ for each $g\in G$.

\smallskip

Let $\Gamma: \cA = \bigoplus_{g\in G} \cA_g$ and $\Gamma':\cA = \bigoplus_{h\in H} \cA'_h$ be two gradings on $\cA$. Then we will say that $\Gamma$ is a \textit{refinement} of $\Gamma'$, or that $\Gamma'$ is a \textit{coarsening} of $\Gamma$, if for any $g \in G$ there exists $h\in H$ such that $\mathcal{A}_g \subseteq \mathcal{A}'_h$; if the inclusion is strict for some $g \in G$, then we will say that we have a \textit{proper} refinement or coarsening. A grading is \textit{fine} if it has no proper refinements.

If $\Gamma$ is a grading on a finite-dimensional algebra $\cA$, a sequence $(n_1, n_2, \dots, n_k)$ of non-negative integers is called the \emph{type} of $\Gamma$ if $n_i$ is the number of homogeneous components of dimension $i$, with $n_k \neq 0$ and $\sum_i n_i = \dim(\cA$).

If $\Gamma: \cA = \bigoplus_{g\in G} \cA_g$ and $\Gamma':\cA = \bigoplus_{h\in H} \cA'_h$ are gradings on $\cA$, we will say that $\Gamma$ and $\Gamma'$ are {\em compatible} if $\cA = \bigoplus_{g\in G, h\in H} \cA_{(g,h)}$ where $\cA_{(g,h)} := \cA_g \cap \cA'_h$. Note that compatibility implies that the subspaces $\cA_{(g,h)}$ define a $G\times H$-grading on $\cA$. A subspace $V$ of $\cA$ is said to be {\em graded} for $\Gamma$ if $V = \bigoplus_{g\in G} (V \cap \cA_g)$.

\smallskip

We can index the homogeneous subspaces of a grading by different groups in such a way that they are still gradings, that is, gradings can be realized by different groups. Therefore, to avoid repetitions in classifications, it is convenient to consider gradings by their universal group. Recall that a group $G_0$ is said to be a {\em universal group} of a grading $\Gamma$ on $\cA$ if it satisfies the following universal property: for any other realization of $\Gamma$ as a $G$-grading, there exists a unique homomorphism $G_0 \to G$ that restricts to the identity on $\supp \Gamma$ (with the natural identification of both supports). It is well-known that universal groups always exist and are unique up to isomorphism \cite[\S~1.2]{EKmon}. We will denote the universal group of $\Gamma$ by $\Univ(\Gamma)$. Note that $\Univ(\Gamma)$ is the group generated by $\supp \Gamma$ with the defining relations $g_1 g_2 = g_3$ for all $g_1, g_2, g_3\in \supp \Gamma$ such that $0\ne\cA_{g_1}\cA_{g_2}\subseteq\cA_{g_3}$.

If $\Gamma$ can be realized as an abelian group grading, then its {\em universal abelian group} can be defined similarly, by adding the restriction of being abelian. In this paper we will only consider gradings by abelian groups, so that additive notation will be used for their products, and by universal group we will mean universal abelian group.

\smallskip

Let $\Gamma$ be a $G$-grading on an algebra $\cA$ and $\Gamma'$ an $H$-grading on an algebra $\cB$. Then $\Gamma$ and $\Gamma'$ are said to be \textit{equivalent} if there exist an isomorphism of algebras $\varphi: \cA \to \cB$ and a bijection $\alpha: \supp\Gamma \to \supp\Gamma'$ such that $\varphi (\cA_g)= \cB_{\alpha(g)}$ for all $g \in \supp\Gamma$.

\smallskip

A degree map related to the universal property that defines $\Univ(\Gamma)$ will be called a \textit{universal degree map} of the $\Univ(\Gamma)$-grading $\Gamma$. Note that $\Gamma$ may be realizable as a $\Univ(\Gamma)$-grading with different degree maps, and some may not satisfy the universal property that defines $\Univ(\Gamma)$. Since $\Univ(\Gamma)$ is unique up to isomorphism, it follows that the universal degree is unique up to equivalence of gradings; in other words, $\deg$ and $\deg'$ are universal degrees of $\Gamma$ as a $\Univ(\Gamma)$-grading if and only if there is $\varphi\in\Aut(\Univ(\Gamma))$ such that $\deg' = \varphi \circ \deg$. When we say that a $G$-grading $\Gamma$, with degree map $\deg$, is {\em given by its universal group}, we usually mean that $G \cong \Univ(\Gamma)$ and $\deg$ is equivalent to the universal degree of $\Gamma$.

\smallskip

If $V$ is a vector space and $G$ an abelian group, a decomposition $\Gamma\colon\; V = \bigoplus_{g\in G} V_g$ will be called a {\em $G$-grading} on $V$. Let $\Gamma'\colon\;W = \bigoplus_{g\in G} W_{g}$ be a $G$-grading on a vector space $W$ and $g\in G$, then a linear map $f\colon V \to W$ is said to be {\em homogeneous of degree $g$} if $f(V_h) \subseteq W_{g+h}$ for each $h\in G$. In that case, $\ker f$ and $\im f$ are graded subspaces for $\Gamma$ and $\Gamma'$, respectively.

\smallskip

Given a $G$-grading $\Gamma$ on $\cA$, the \emph{automorphism group of $\Gamma$}, $\Aut(\Gamma)$, is the group of self-equivalences of $\Gamma$. The \emph{stabilizer of $\Gamma$}, $\Stab(\Gamma)$, is the group of $G$-automorphisms of $\Gamma$, i.e., the group of automorphisms of $\cA$ that fix the homogeneous components. The \emph{diagonal group of $\Gamma$}, $\Diag(\Gamma)$, is the subgroup of $\Stab(\Gamma)$ consisting of the automorphisms that act by multiplication by a nonzero scalar on each homogeneous component. The \emph{Weyl group of $\Gamma$} is the quotient group $\cW(\Gamma)=\Aut(\Gamma)/\Stab(\Gamma)$, which can be regarded as a subgroup of $\Sym(\supp\Gamma)$ and of $\Aut(\Univ(\Gamma))$.

\smallskip

In this paper, fine gradings will be classified up to equivalence and by their universal abelian groups. Recall that for finite-dimensional algebras, the non-fine gradings can be obtained by computing the coarsenings of the fine gradings, and these are induced by quotients of the universal groups.

The definitions of affine group schemes and their relation with gradings can be consulted in \cite[Appendix A]{EKmon}.

\subsection{Kantor pairs and Kantor triple systems}

Let $(\cA,\inv)$ be an $\FF$-algebra with involution. Then $\cA=\cH(\cA,\inv)\oplus\cS(\cA,\inv)$, where:
$$
\cH(\cA,\inv)=\{a\in\cA \med \bar{a}=a\}\;\; \mbox{and} \;\; \cS(\cA,\inv)=\{a\in\cA \med \bar{a}=-a\}.
$$
The subspaces $\cH(\cA,\inv)$ and $\cS(\cA,\inv)$ are called, respectively, the \textit{hermitian} (or \textit{symmetric}) subspace and the \textit{skew-symmetric} subspace.

\begin{df}
A unital $\FF$-algebra with involution $(\cA,\inv)$ is said to be {\it structurable} if
\begin{equation}\label{struc}
[V_{x,y},V_{z,w}] = V_{V_{x,y}z,w}-V_{z,V_{y,x}w}\quad\mbox{for all}\;x,y,z,w\in\cA,
\end{equation}
where $ V_{x,y}(z) = \{x,y,z\}:= (x\bar y)z + (z\bar y)x - (z\bar x)y $.
The $U$-operator is defined by $U_{x,z}(y) := \{x,y,z\}$ and $U_x := U_{x,x}$.
\end{df}
Note that, in the literature, fields of characteristic $3$ are usually excluded in the definition of structurable algebra; this is due to problems arising in characteristic $3$, including some results of general theory that may not hold in that case. However, here we will include that case.

Recall that the associative center of $\cA$, denoted by $Z(\cA)$, is the set of elements $z\in\cA$ satisfying the equalities $xz=zx$ and $(z,x,y)=(x,z,y)=(x,y,z)=0$ for all $x,y\in\cA$, where $(x,y,z) := (xy)z - x(yz)$ denotes the associator. The {\it center} of $(\cA,\inv)$ is defined by $Z(\cA,\inv) = Z(\cA)\cap\cH(\cA,\inv)$, and a structurable algebra $\cA$ is said to be {\it central} if $Z(\cA,\inv)=\FF1$.

Now we recall the well-known classification of central simple structurable algebras:

\begin{theorem}[Allison, Smirnov] If $\chr\FF \neq 2, 3, 5$, then any finite-dimensional central simple structurable $\FF$-algebra belongs to one of the following six (non-disjoint) classes:
\begin{itemize}
\item[(1)] central simple associative algebras with involution,
\item[(2)] central simple Jordan algebras (with identity involution),
\item[(3)] structurable algebras constructed from a non-degenerate Hermitian form
over a central simple associative algebra with involution,
\item[(4)] forms of the tensor product of two Hurwitz algebras,
\item[(5)] simple structurable algebras of skew-dimension 1 (forms of structurable matrix algebras),
\item[(6)] an exceptional 35-dimensional case (Smirnov algebra), which can be constructed from an
octonion algebra.\qed
\end{itemize}
\end{theorem}
This classification was given by Allison in the case of characteristic $0$ (\cite{Al78}), where case (6) was overlooked. Smirnov completed the classification and gave the generalization for the case with $\chr\FF\neq 2,3,5$ (\cite{Smi92}).

\begin{df}
A {\em Kantor pair} (or {\em generalized Jordan pair of second order} \cite{F94, AF99}) is a pair of vector spaces $\cV = (\cV^+, \cV^-)$ and a pair of trilinear products $\cV^\sigma\times\cV^{-\sigma}\times\cV^\sigma\to\cV^\sigma$, denoted by $\{x,y,z\}^\sigma$, satisfying the identities:
\begin{align}
& [V^\sigma_{x,y}, V^\sigma_{z,w}] = V^\sigma_{V^\sigma_{x,y}z,w} - V^\sigma_{z,V^{-\sigma}_{y,x}w}, \\
& K^\sigma_{K^\sigma_{x,y}z,w} = K^\sigma_{x,y}V^{-\sigma}_{z,w} + V^\sigma_{w,z}K^\sigma_{x,y},
\end{align}
where $V^\sigma_{x,y}z = U^\sigma_{x,z}(y):= \{x,y,z\}^\sigma$, $U^\sigma_x := U^\sigma_{x,x}$ and $K^\sigma_{x,y}z = K^\sigma(x,y)z := \{x,z,y\}^\sigma - \{y,z,x\}^\sigma$. The map $V^\sigma_{x,y}$ is sometimes denoted by $D^\sigma_{x,y}$ or $D^\sigma(x,y)$, because $(V^+_{x,y},-V^-_{y,x})$ is a derivation of the Kantor pair.
The superscript $\sigma$ will always take the values $+$ and $-$, and may be omitted when there is no ambiguity.
\end{df}

\begin{df}
A {\em Kantor triple system} (or {\em generalized Jordan triple system of second order} \cite{K72, K73}) is a vector space $\cT$ with a trilinear product $\cT\times\cT\times\cT\to\cT$, denoted by $\{x,y,z\}$, which satisfies:
\begin{align} \label{}
& [V_{x,y}, V_{z,w}] = V_{V_{x,y}z,w} - V_{z,V_{y,x}w}, \\
& K_{K_{x,y}z,w} = K_{x,y}V_{z,w} + V_{w,z}K_{x,y},
\end{align}
where $V_{x,y}z = U_{x,z}(y) := \{x,y,z\}$, $U_x := U_{x,x}$ and $K_{x,y}z := \{x,z,y\} - \{y,z,x\}$.
\end{df}

Given a structurable algebra $\cA$, we can define its associated Kantor triple system as the vector space $\cA$ endowed with the triple product $\{x,y,z\}$ of $\cA$. Similarly, with two copies of a Kantor triple system $\cT$ and two copies of its triple product we can define the associated Kantor pair $\cV = (\cT, \cT)$. In particular, a structurable algebra $\cA$ with its triple product defines a Kantor pair $(\cA, \cA)$. Note that Jordan pairs (respectively, Jordan triple systems) are particular cases of Kantor pairs (respectively, Kantor triple systems); they are exactly those where $K_{x,y}=0$ for all $x,y$.

\bigskip

We will now recall the Kantor construction from a Kantor pair, which produces a $5$-graded Lie algebra. The Kantor construction is a generalization of the Tits-Kantor-Koecher (TKK) construction from Jordan pairs. Consider the vector space
\begin{equation}
\kan(\cV) := \kan(\cV)^{-2} \oplus \kan(\cV)^{-1} \oplus \kan(\cV)^{0} \oplus \kan(\cV)^{1} \oplus \kan(\cV)^{2}, 
\end{equation}
where
\begin{align*}
& \kan(\cV)^{-2} = \left(\begin{matrix} 0 & K(\cV^-,\cV^-) \\ 0 & 0 \end{matrix}\right), \quad
\kan(\cV)^{-1} = \left(\begin{matrix} \cV^- \\ 0 \end{matrix}\right), \\
& \kan(\cV)^{0} = \lspan\left\{\left(\begin{matrix} D(x^-,x^+) & 0 \\ 0 & -D(x^+,x^-) \end{matrix}\right) 
  \med x^\sigma\in\cV^\sigma \right\}, \\
& \kan(\cV)^{1} = \left(\begin{matrix} 0 \\ \cV^+ \end{matrix}\right), \quad
\kan(\cV)^{2} = \left(\begin{matrix} 0 & 0 \\ K(\cV^+,\cV^+) & 0 \end{matrix}\right).
\end{align*}
Then, the vector space
\begin{align*}
\mathfrak{S}(\cV) :=& \kan(\cV)^{-2} \oplus \kan(\cV)^0 \oplus \kan(\cV)^{2} \\
=& \lspan\left\{\left(\begin{matrix} D(x^-,x^+) & K(y^-,z^-) \\ K(y^+,z^+) & -D(x^+,x^-) \end{matrix}\right) 
  \med x^\sigma,y^\sigma,z^\sigma\in\cV^\sigma \right\}
\end{align*}
is a subalgebra of the Lie algebra 
$$\End\begin{pmatrix} \cV^- \\ \cV^+ \end{pmatrix}
= \begin{pmatrix} \End(\cV^-) & \Hom(\cV^+,\cV^-) \\ \Hom(\cV^-,\cV^+) & \End(\cV^+) \end{pmatrix} ,$$
with the commutator product. Now define an anti-commutative product on $\kan(\cV)$ by means of
\begin{align*}
& [A,B] = AB-BA, \qquad 
[A, \begin{pmatrix} x^- \\ x^+ \end{pmatrix} ] = A \begin{pmatrix} x^- \\ x^+ \end{pmatrix}, \\
& [\begin{pmatrix} x^- \\ x^+ \end{pmatrix}, \begin{pmatrix} y^- \\ y^+ \end{pmatrix}] =
\left(\begin{matrix} D(x^-,y^+)-D(y^-,x^+) & K(x^-,y^-) \\ K(x^+,y^+) & -D(y^+,x^-)+D(x^+,y^-) \end{matrix}\right)
\end{align*}
where $x^{\sigma},y^{\sigma} \in \cV^{\sigma}$ and $A,B\in \mathfrak{S}(\cV)$.
Then $\kan(\cV)$ becomes a Lie algebra, called the \emph{Kantor Lie algebra} of $\cV$. The $5$-grading (i.e., the grading with support $\{0,\pm 1, \pm 2\}$) is a $\ZZ$-grading which is called the \emph{standard grading} on $\kan(\cV)$, but we will also refer to it as the \emph{main grading} on $\kan(\cV)$. The subspaces $\kan(\cV)^1$ and $\kan(\cV)^{-1}$ are usually identified with $\cV^+$ and $\cV^-$, respectively. The Kantor construction of a structurable algebra or Kantor triple system is defined as the Kantor construction of the associated Kantor pair.

Conversely, it is well-known that a $5$-graded Lie algebra $L = \bigoplus_{i\in\ZZ} L_i$ produces a Kantor pair $\cV = (L_{-1}, L_1)$ with triple products defined by
$$\{x^\sigma, y^{-\sigma}, z^\sigma\} := [[x^\sigma, y^{-\sigma}], z^\sigma].$$

\smallskip

Let $\cA$ be a structurable algebra and $\cV = (\cV^+, \cV^-) = (\cA, \cA)$ the associated Kantor pair. Recall that $\nu(x^-,x^+):=(D_{x^-,x^+},-D_{x^+,x^-})$ is a derivation called \emph{inner derivation} associated with $(x^-,x^+)\in\cV^-\times\cV^+$. The \emph{inner structure algebra} of $\cA$ is the Lie algebra $\mathfrak{innstr}(\cA) :=\lspan\{\nu(x,y)\med x,y\in\cA\}$. Let $L_x$ denote the left multiplication by $x\in\cA$ and write $\cS=\cS(\cA)$. Then, the map $\cS\to L_\cS$, $s\mapsto L_s$, is a linear monomorphism, thus we can identify $\cS$ with $L_\cS$. Also, note that the map $\cA\times\cA\to\cS$ given by $\psi(x,y):=x\bar y-y\bar x$ is an epimorphism (because $\psi(s,1)=2s$ for $s\in\cS$). By \cite[(1.3)]{AF84}, we have the identity $L_{\psi(x,y)}=U_{x,y}-U_{y,x}=K(x,y)$ for all $x,y\in\cA$. Consequently, we can identify the subspaces $\kan(\cV)^{2}$ and $\kan(\cV)^{-2}$ with $L_\cS$, and also with $\cS$. This allows to write the main grading on $\kan(\cA)$ in a well-known second form, as follows:
\begin{equation}\label{kantor.construction}
\kan(\cA) = \cS^- \oplus \cA^- \oplus \mathfrak{innstr}(\cA) \oplus \cA^+ \oplus \cS^+.
\end{equation}

This construction can be used to induce gradings on the Kantor Lie algebra from gradings on a structurable algebra, Kantor pair, or Kantor triple system. That is one of the aims of this work, where the particular case related to Hurwitz algebras is studied.

Finally, note that for a Kantor pair $\cV$, each automorphism $\varphi\in\Aut(\cV)$ has a natural extension $\widetilde{\varphi} \in \Aut(\kan(\cV))$ whose action on $\mathfrak{S}(\cV)$ is given by
\begin{equation}
A = \begin{pmatrix} A_{11} & A_{12} \\ A_{21} & A_{22} \end{pmatrix} \mapsto
\begin{pmatrix} \varphi^- A_{11} (\varphi^-)^{-1} & \varphi^- A_{12} (\varphi^+)^{-1} \\
                \varphi^+ A_{21} (\varphi^-)^{-1} & \varphi^+ A_{22} (\varphi^+)^{-1} \end{pmatrix} 
\end{equation}
for each $A\in \mathfrak{S}(\cV)$, and we have that $\AAut(\cV) \leq \AAut(\kan(\cV))$.

\subsection{Gradings on Kantor pairs and Kantor triple systems}

\begin{df}
Let $G$ be an abelian group and $\cV$ a Kantor pair. Given two decompositions of vector spaces $\Gamma^{\sigma} \colon \cV^\sigma = \bigoplus_{g\in G} \cV_g^{\sigma}$, for $\sigma = \pm$, we will say that $\Gamma=(\Gamma^+,\Gamma^-)$ is a {\em $G$-grading on $\cV$} if $\lbrace \cV^{\sigma}_g, \cV^{-\sigma}_h, \cV^{\sigma}_k \rbrace \subseteq \cV^{\sigma}_{g+h+k}$ for any $g,h,k\in G$ and $\sigma\in \lbrace +,- \rbrace$. The vector space $\cV^+_g \oplus \cV^-_g$ is the {\em homogeneous component of degree} $g$. If $0\neq x\in \cV^{\sigma}_g$ we say $x$ is {\em homogeneous of degree $g$} and we write $\deg(x)=g$. 

Let $G$ be an abelian group and $\cT$ a Kantor triple system. A {\em $G$-grading on $\cT$} is a decomposition of vector spaces $\Gamma \colon \cT = \bigoplus_{g\in G} \cT_g$ such that $\lbrace \cT_g, \cT_h, \cT_k \rbrace \subseteq \cT_{g+h+k}$ for any $g,h,k\in G$. Definitions of homogeneous components, homogeneous elements and degree are analogous to the case of Kantor pairs.

The rest of definitions related to gradings (universal group, equivalence and isomorphism of gradings, etc.) are analogous to the algebra case (see \cite[Chapter~1]{EKmon}).
\end{df}

Let $\Gamma$ be a $G$-grading on a Kantor pair $\cV$ with degree $\deg$. Fix $g\in G$. For any homogeneous elements $x^+\in \cV^+$ and $y^-\in \cV^-$, set $\deg_g(x^+):=\deg(x^+)+g$, $\deg_g(y^-):=\deg(y^-)-g$. This defines a new $G$-grading, which will be denoted by $\Gamma^{[g]}$ and called the {\em $g$-shift of $\Gamma$}. Note that, although $\Gamma$ and $\Gamma^{[g]}$ may fail to be equivalent (because the shift may collapse or split a homogeneous subspace of $\cV^+$ with another of $\cV^-$), the intersection of their homogeneous components with $\cV^\sigma$ coincide for each $\sigma$. It is clear that $(\Gamma^{[g]})^{[h]}=\Gamma^{[g+h]}$. Similarly, if $\Gamma$ is a $G$-grading on a Kantor triple system $\cT$ and $g\in G$ has order $1$ or $2$, we can define the {\em $g$-shift} $\Gamma^{[g]}$ with the new degree $\deg_g(x) := \deg(x) + g$.

A $G$-grading on $\kan(\cV)$ is called {\em Kantor-compatible} if $\kan(\cV)^{-1}$ and $\kan(\cV)^1$ are $G$-graded spaces. Note that in this case $\kan(\cV)^0=\lspan [\kan(\cV)^1,\kan(\cV)^{-1}]$, $\kan(\cV)^2=\lspan [\kan(\cV)^1,\kan(\cV)^{1}]$ and $\kan(\cV)^{-2}=\lspan [\kan(\cV)^{-1},\kan(\cV)^{-1}]$ are graded too. In other words, a grading on $\kan(\cV)$ is Kantor-compatible if and only if it is compatible with the $\ZZ$-grading associated with the Kantor construction. A $G$-grading $\Gamma$ on $\cV$ can be extended to a Kantor-compatible $G$-grading $E_G(\Gamma):\kan(\cV)=\bigoplus_{g\in G} \kan(\cV)_g$ by setting 
\begin{align*}
& \kan(\cV)^{-2}_g = \lspan\left\{\left(\begin{matrix} 0 & K(\cV^-_{g_1},\cV^-_{g_2}) \\ 0 & 0 \end{matrix}\right)\med g_1+g_2=g  \right\}, \quad
\kan(\cV)^{-1}_g = \left(\begin{matrix} \cV^-_g \\ 0 \end{matrix}\right), \\
& \kan(\cV)^{0}_g = \lspan\left\{\left(\begin{matrix} D(x^-,x^+) & 0 \\ 0 & -D(x^+,x^-) \end{matrix}\right) 
  \med x^{\sigma}\in \cV^{\sigma}, \mathrm{deg}(x^+)+\mathrm{deg}(x^-)=g \right\}, \\
& \kan(\cV)^{1}_g = \left(\begin{matrix} 0 \\ \cV^+_g \end{matrix}\right), \quad
\kan(\cV)^{2}_g = \lspan\left\{\left(\begin{matrix} 0 & 0 \\ K(\cV^+_{g_1},\cV^+_{g_2}) & 0 \end{matrix}\right)\med g_1+g_2=g  \right\}.
\end{align*}
Note that we have $\kan(\cV)_g =\bigoplus_{i=-2}^{2} \kan(\cV)_g^i$ where $\kan(\cV)_g^i \coloneqq \kan(\cV)^i \cap \kan(\cV)_g$. 

Conversely, any Kantor-compatible $G$-grading $\widetilde{\Gamma}$ on $\kan(\cV)$ restricts to a $G$-grading $R_G(\widetilde{\Gamma})$ on $\cV$, because $\lbrace x^{\sigma}, y^{-\sigma}, z^{\sigma} \rbrace= [[x^{\sigma}, y^{-\sigma}],z^{\sigma}]$ for $x^{\sigma},z^{\sigma}\in \cV^{\sigma}$, $y^{-\sigma}\in \cV^{-\sigma}$ and $\sigma\in\lbrace+,-\rbrace$. (We identify $x^-$ with $\left(\begin{matrix} x^- \\ 0 \end{matrix}\right)$ and $x^+$ with $\left(\begin{matrix} 0 \\ x^+ \end{matrix}\right)$ for all $x^{\sigma}\in \cV^{\sigma}$.)

Denote by $\mathrm{Grad}_G(\cV)$ the set of $G$-gradings on $\cV$, and by $\mathrm{KGrad}_G(\kan(\cV))$ the set of Kantor-compatible $G$-gradings on $\kan(\cV)$. We will call $E_G:\mathrm{Grad}_G(\cV) \rightarrow \mathrm{KGrad}_G(\kan(\cV))$ the {\em extension map} and $R_G:\mathrm{KGrad}_G(\kan(\cV))\rightarrow\mathrm{Grad}_G(\cV)$ the {\em restriction map}.

The following result is an extension of the Jordan case in \cite[Th.~2.13]{Ara17}.

\begin{proposition} \label{gradings.correspondence.from.pairs}
Let $\cV$ be a Kantor pair with associated Lie algebra $\kan(\cV)$ and let $G$ be an abelian group. Then, the maps $E_G$ and $R_G$ are inverses of each other. Coarsenings are preserved by the correspondence, that is, given a $G_i$-grading $\Gamma_i$ on $\cV$ with extended $G_i$-grading $\widetilde{\Gamma}_i=E_{G_i}(\Gamma_i)$ on $\kan(\cV)$ for $i=1,2$, and a homomorphism $\alpha:G_1\rightarrow G_2$, then $\Gamma_2={^{\alpha}}\Gamma_1$ if and only if $\widetilde{\Gamma}_2={^{\alpha}}\widetilde{\Gamma}_1$. Let $\Gamma$ be a $G$-grading on $\cV$. If $G = \Univ(\Gamma)$, then $G = \Univ(E_G(\Gamma))$. Moreover, $\Gamma$ is fine and $G = \Univ(\Gamma)$ if and only if $E_G(\Gamma)$ is fine and $G = \Univ(E_G(\Gamma))$.
\end{proposition}

\begin{proof}
By construction $E_G$ and $R_G$ are inverses of each other.

Assume that $\Gamma_2={^{\alpha}}\Gamma_1$ for some homomorphism $\alpha:G_1\rightarrow G_2$.  Since $\cV^{\pm 1}_g\subseteq \cV^{\pm 1}_{\alpha(g)}$ for any $g\in G$, we have $\kan(\cV)_g^{\pm 1}\subseteq \kan(\cV)_{\alpha(g)}^{\pm 1}$ and 
\[\kan(\cV)_g^{i+j}=\sum_{g_1+g_2=g}[\kan(\cV)_{g_1}^i,\kan(\cV)_{g_2}^j] \subseteq \sum_{g_1+g_2=g}[\kan(\cV)_{\alpha(g_1)}^i,\kan(\cV)_{\alpha(g_2)}^j] \subseteq \kan(\cV)^{i+j}_{\alpha(g)}\]
for $i,j\in\lbrace 1,-1\rbrace$. Then $\kan(\cV)_{g}\subseteq \kan(\cV)_{\alpha(g)}$ for any $g\in G$. Hence $\widetilde{\Gamma}_1$ refines $\widetilde{\Gamma}_2$ and $\widetilde{\Gamma}_2 = {^{\alpha}}\widetilde{\Gamma}_1$. Conversely, if $\widetilde{\Gamma}_2 = {^{\alpha}}\widetilde{\Gamma}_1$, by restriction we obtain $\Gamma_2={^{\alpha}}\Gamma_1$. This proves that the coarsenings are preserved.

Consider $\widetilde{\Gamma}=E_G(\Gamma)$ with $G=\Univ(\Gamma)$. Note that $\supp \Gamma$ generates $\Univ(\Gamma)$ and $\Univ(\widetilde{\Gamma})$. Since the $\Univ(\widetilde{\Gamma})$-grading $\widetilde{\Gamma}$ restricts to $\Gamma$ as a $\Univ(\widetilde{\Gamma})$-grading, there is a unique homomorphism $G=\Univ(\Gamma)\rightarrow \Univ(\widetilde{\Gamma})$ which restricts to the identity in $\supp (\Gamma)$. Conversely, $\Gamma$ extends to $\widetilde{\Gamma}$ as a $G$-grading, so there is a unique homomorphism $\Univ(\widetilde{\Gamma}) \rightarrow G$ which restricts to the identity map in $\supp(\widetilde{\Gamma})$ and in $\supp(\Gamma)$. It is clear that $\Univ(\widetilde{\Gamma})=\langle \supp \Gamma \rangle$. Therefore the compositions $G \rightarrow \Univ(\widetilde{\Gamma})\rightarrow G$ and $\Univ(\widetilde{\Gamma})\rightarrow G \rightarrow \Univ(\widetilde{\Gamma})$ are the identity map, and $G=\Univ(\widetilde{\Gamma})$.

Suppose again that $\widetilde{\Gamma}=E_G(\Gamma)$. Note that $\Gamma$ is a fine $G$-grading on $\cV$ with $G=\Univ(\Gamma)$ if and only if $\supp \Gamma$ generates $G$ and $\Gamma$ satisfies the following property: if $\Gamma= {^{\alpha}}\Gamma_0$ for some $G_0$-grading $\Gamma_0$ on $\cV$, where $G_0$ is generated by $\supp \Gamma_0$ and $\alpha:G_0\rightarrow G$ is an epimorphism, then $\alpha$ is an isomorphism. The same is true for Kantor-compatible gradings. Since the coarsenings are preserved in the correspondence, so does this property (also note that $\supp\Gamma$ generates $G$ if and only if $\supp\widetilde{\Gamma}$ generates $G$). Then we get that $\Gamma$ is fine and $G=\Univ(\Gamma)$ if and only if $E_G(\Gamma)$ is fine in the class of Kantor-compatible gradings and $G=\Univ(E_G(\Gamma))$. Moreover, if $\widetilde{\Gamma}$ is fine in the class of Kantor-compatible gradings, then the supports of $\kan(\cV)^{i}$ for $i=-2,-1,0,1,2$ are disjoint and therefore  $\widetilde{\Gamma}$ is also fine in the class of all abelian group gradings on $\kan(\cV)$. Therefore $\Gamma$ is fine and $G=\Univ(\Gamma)$ if and only if $E_G(\Gamma)$ is fine and $G=\Univ(E_G(\Gamma))$.
\end{proof}

Note that if $\Gamma$ is a grading on a Kantor pair $\cV$ given by its universal group $G = \Univ(\Gamma)$, since automorphisms extend from $\cV$ to $\kan(\cV)$, we have that $\cW(\Gamma) \leq \cW(E_G(\Gamma))$.

\subsection{Gradings on Hurwitz algebras}

Recall that for a Hurwitz algebra $C$, the equation 
\begin{equation}
x^2 - n(x,1)x +n(x)1 =0\label{eq_quadratic_Hurwitz} \\
\end{equation}
 is satisfied by any $x\in C$, where $n$ denotes both the norm of $C$ and its polar form. The norm is nondegenerate and multiplicative ($n(xy)=n(x)n(y)$ for any $x,y \in C$), and the \textit{trace} of $x\in C$ is defined by $t(x)\coloneqq n(x,1)$.
It was proven in \cite{AM99} that the split Cayley algebra $\cC$ has a homogeneous basis $B_\CD(\cC) \coloneqq \{x_g\}_{g\in\ZZ_2^3}$ associated with a (fine) $\ZZ_2^3$-grading, with degree map $\deg(x_g) \coloneqq g$, and such that the product is given by
\begin{equation}
x_g x_h \coloneqq \sigma(g,h) x_{g+h}
\end{equation}
with
\begin{align*}
& \sigma(g,h) = \sigma_{g, h} \coloneqq (-1)^{\psi(g,h)}, \\
& \psi(g,h) \coloneqq h_1 g_2 g_3 + g_1 h_2 g_3 + g_1 g_2 h_3 + \sum_{i\leq j} g_i h_j,
\end{align*}
for any $g = (g_1, g_2, g_3)$, $h = (h_1, h_2, h_3) \in \ZZ_2^3$.
The multiplication constants of this basis satisfy the property $\sigma_{g, h+k} = \sigma_{g,h}\sigma_{g,k}$ for any $g,h,k\in\ZZ_2^3$; in other words, $\sigma_g \in \widehat{\ZZ}_2^3$ for all $g\in \ZZ_2^3$, where we denote $\sigma_g(h) := \sigma_{g,h}$ for $h\in \ZZ_2^3$. (Here, $\widehat{G}$ denotes the character group of an abelian group $G$.) Moreover, $x_0 = 1$, $B_\CD(\cC)$ is an orthonormal basis relative to the norm, and the involution is given by $\bar x_g = \sigma_{g,g} x_g$ for any $g\in\ZZ_2^3$. We will refer to $B_\CD(\cC)$ as a {\it Cayley-Dickson basis} of $\cC$.

Furthermore, for each subgroup $\ZZ_2^m \cong H \leq \ZZ_2^3$ with $m\in\{0,1,2\}$, we have that $C \coloneqq \lspan\{ x_h \}_{h\in H}$ is an $H$-graded Hurwitz algebra of dimension $2^m$, and the basis $B_\CD(C)$ of $C$ inherits the good properties from the basis $B_\CD(\cC)$; we will say that $B_\CD(C)$ is a {\it Cayley-Dickson basis} of $C$.

Let $\{a_i\}_{i=1}^3$ be the canonical basis of $\ZZ_2^3$.
Then, if we consider $\ZZ_2^3$ with the order given by
$$(0, a_1, a_2, a_1 + a_2, a_3, a_1 + a_3, a_2 + a_3, a_1 + a_2 + a_3),$$
the multiplication constants for the Cayley-Dickson basis on $\cC$ are given by
\begin{equation} \label{sigmamatrix}
(\sigma_{g,h})_{g,h\in\ZZ_2^3} = \left[\begin{array}{rrrrrrrr}
\multicolumn{1}{r|}{1} & \multicolumn{1}{r|}{1} & 1 & \multicolumn{1}{r|}{1} & 1 & 1 & 1 & 1 \\
\cline{1-1}
1 & \multicolumn{1}{r|}{-1} &  1 & \multicolumn{1}{r|}{-1} & -1 &  1 & -1 & 1 \\
\cline{1-2}
1 & -1 & -1 & \multicolumn{1}{r|}{1}  & -1 &  1 &  1 & -1 \\
1 &  1 & -1 & \multicolumn{1}{r|}{-1} & -1 & -1 &  1 &  1 \\
\cline{1-4}
1 &  1 &  1 &  1   &   -1 & -1 & -1 & -1 \\
1 & -1 & -1 &  1   &    1 & -1 & -1 &  1 \\
1 &  1 & -1 & -1   &    1 &  1 & -1 & -1 \\
1 & -1 &  1 & -1   &    1 & -1 &  1 & -1
\end{array}\right],
\end{equation}
where the highlighted submatrices above correspond to the constants for the cases $\dim C = 1,2,4$.

\bigskip

There is a basis $B_\ZZ(\cC) \coloneqq \{e_1,e_2,u_1,u_2,u_3,v_1,v_2,v_3\}$ of the split Cayley algebra $\cC$ where the multiplication is given as in
Figure~\ref{product.cartan.basis} (\cite[\S4.1]{EKmon}). We will call this basis a {\it Cartan basis} of $\cC$.

\begin{figure}[htbp]
\abovetabulinesep = 1.5mm
\belowtabulinesep = 1.5mm
\begin{tabu}{c|[1pt]rrrrrrrr}
 & $e_1$ & $e_2$ & $u_1$ & $u_2$ & $u_3$ & $v_1$ & $v_2$ & $v_3$ \\
 \tabucline[1pt]{1-9}
 $e_1$ & $e_1$ & \multicolumn{1}{r|}{0} & $u_1$ & $u_2$ & \multicolumn{1}{r|}{$u_3$} & 0 & 0 & \multicolumn{1}{r|[1pt]}{0} \\[-1pt]
 $e_2$ & 0 & \multicolumn{1}{r|}{$e_2$} & 0 & 0 & \multicolumn{1}{r|}{0} & $v_1$ & $v_2$ & \multicolumn{1}{r|[1pt]}{$v_3$} \\[-1pt]
 \tabucline{2-9}
 $u_1$ & 0 & \multicolumn{1}{r|}{$u_1$} & 0 & $v_3$ & \multicolumn{1}{r|}{$-v_2$} & $-e_1$ & 0 & \multicolumn{1}{r|[1pt]}{0} \\[-1pt]
 $u_2$ & 0 & \multicolumn{1}{r|}{$u_2$} & $-v_3$ & 0 & \multicolumn{1}{r|}{$v_1$} & 0 & $-e_1$ & \multicolumn{1}{r|[1pt]}{0} \\[-1pt]
 $u_3$ & 0 & \multicolumn{1}{r|}{$u_3$} & $v_2$ & $-v_1$ & \multicolumn{1}{r|}{0} & 0 & 0 & \multicolumn{1}{r|[1pt]}{$-e_1$} \\[-1pt]
 \tabucline{2-9}
 $v_1$ & $v_1$ & \multicolumn{1}{r|}{0} & $-e_2$ & 0 & \multicolumn{1}{r|}{0} & 0 & $u_3$ & \multicolumn{1}{r|[1pt]}{$-u_2$} \\[-1pt]
 $v_2$ & $v_2$ & \multicolumn{1}{r|}{0} & 0 & $-e_2$ & \multicolumn{1}{r|}{0} & $-u_3$ & 0 & \multicolumn{1}{r|[1pt]}{$u_1$} \\[-1pt]
 $v_3$ & $v_3$ & \multicolumn{1}{r|}{0} & 0 & 0 & \multicolumn{1}{r|}{$-e_2$} & $u_2$ & $-u_1$ & \multicolumn{1}{r|[1pt]}{0} \\[-1pt]
 \tabucline[1pt]{2-9}
\end{tabu}
\caption{Multiplication for the Cartan basis}
\label{product.cartan.basis}
\end{figure}
If $C$ is a Hurwitz algebra with $\dim C = 1,2,4$, respectively, then we define the \emph{Cartan basis} $B_\ZZ(C)$ of $C$ as $\{e_1\}$, $\{e_1, e_2\}$ or $\{e_1, e_2, u_1, v_1\}$, respectively, with the same products as in Figure~\ref{product.cartan.basis}.

Recall from \cite[\S4.1]{EKmon} that there is a fine $\ZZ^2$-grading on $\cC$, where the Cartan basis is homogeneous and the degree map is given by

\begin{equation} \label{cartan.degree.C}
\begin{aligned}
& \deg(e_1)=(0,0)=\deg(e_2), \; & \deg(u_1)=(1,0)=-\deg(v_1), \\
&\deg(u_2)=(0,1)=-\deg(v_2), \; & \deg(v_3)=(1,1)=-\deg(u_3).
\end{aligned}
\end{equation}
If $\dim C=4$, we have a fine $\ZZ$-grading on $C$ where the Cartan basis is homogeneous with degree map given by $\deg(e_i)=0$ for $i=1,2$ and $\deg(u_1)=1=-\deg(v_1)$. For the cases where $\dim C=1,2$, the only grading on $C$ where the Cartan basis is homogeneous is the trivial grading (see \cite[Remark 4.16]{EKmon}).

\bigskip

It is straightforward to see that we can obtain a Cayley-Dickson basis from a Cartan basis via the expressions:
\begin{equation} \label{cd.from.cartan}
\begin{aligned}
x_0 &= e_1 + e_2, & x_{a_1} &= -\bi(e_1 - e_2), \\
x_{a_2} &= u_1 + v_1, & x_{a_1 + a_2} &= -\bi(u_1 - v_1), \\
x_{a_3} &= u_2 + v_2, & x_{a_1 + a_3} &= \bi(u_2 - v_2), \\
x_{a_2 + a_3} &= -(u_3 + v_3), & x_{a_1 + a_2 + a_3} &= -\bi(u_3 - v_3),
\end{aligned}
\end{equation}
where $\bi$ is a square root of $-1$ in $\FF$, or equivalently:
\begin{equation} \label{cartan.from.cd}
\begin{aligned}
e_1 &= (x_0 + \bi x_{a_1})/2, & e_2 &= (x_0 - \bi x_{a_1})/2, \\
u_1 &= (x_{a_2} + \bi x_{a_1+a_2})/2, & v_1 &= (x_{a_2} - \bi x_{a_1+a_2})/2, \\
u_2 &= (x_{a_3} - \bi x_{a_1+a_3})/2, & v_2 &= (x_{a_3} + \bi x_{a_1+a_3})/2, \\
u_3 &= -(x_{a_2+a_3} - \bi x_{a_1+a_2+a_3})/2, & v_3 &= -(x_{a_2+a_3} + \bi x_{a_1+a_2+a_3})/2. \\
\end{aligned}
\end{equation}

\begin{notation} \label{NotationCartan}
In order to study gradings on Kantor pairs and triple systems in further sections, it may be convenient to introduce the following notation, which is better suited to describe the new symmetries appearing.

Let $\bi$ be a square root of $-1$ in $\FF$. Consider the parity operator $|a| \coloneqq \text{ord}(a) - 1 \in \{0, 1\}$ for $a\in\ZZ_2$; also note that 
\begin{equation} \label{bi.equation}
\bi^{|a| + |b|} = \sigma_{a, b} \bi^{|a + b|} \quad \text{and} \quad (-1)^{|a|} = \sigma_{a,a}
\end{equation}
for $a,b\in\ZZ_2$, where $\sigma_{a,b}$ denotes the corresponding multiplication constant for a Cayley-Dickson basis of the $2$-dimensional Hurwitz algebra.
Let $C$ be a Hurwitz algebra with $\dim C = 2^m > 1$ and consider a Cayley-Dickson basis $B_\CD(C) = \{x_g\}_{g\in\ZZ_2^m}$. Using the same basis $\{a_i\}_{i=1}^m$ associated with \eqref{sigmamatrix}, we will denote $x_g^a \coloneqq x_{(g, a)}$ for $(g, a) \in \ZZ_2^{m-1} \times \ZZ_2 \equiv \ZZ_2^m$, where the group $\ZZ_2$ is generated by $a_1$, and on the other hand, $\ZZ_2^{m-1}$ is generated by $\{a_2, a_3\}$ if $m = 3$ and by $\{a_2\}$ if $m = 2$. It is also straightforward to see that
\begin{equation} \label{sigma.other.equation}
\sigma_{a,k} \sigma_{k,a} = (-1)^{|a||k|} = \sigma_{a,a}^{|k|} = \sigma_{k,k}^{|a|}
\end{equation}
for each $(k, a) \in \ZZ_2^{m-1} \times \ZZ_2$. We will denote
\begin{equation} \label{second.cartan.basis}
v^\alpha_g \coloneqq \frac{\sqrt{2}}{4} \sum_{a\in\ZZ_2} \alpha(a) \sigma_{a, g} \bi^{|a|} x_g^a,
\end{equation}
for each $g\in\ZZ_2^{m-1}$ and $\alpha\in\widehat{\ZZ}_2$. We can recover the original Cayley-Dickson basis via the expression:
\begin{equation}
x^a_g = \sqrt{2} \sum_{\alpha\in\widehat{\ZZ}_2} \alpha(a) \sigma_{a, g} (-\bi)^{|a|} v_g^\alpha,
\end{equation}
for each $(g, a)\in\ZZ_2^{m-1} \times \ZZ_2$. Indeed,
\begin{equation*} \begin{aligned}
\sqrt{2} & \sum_{\alpha\in\widehat{\ZZ}_2} \alpha(a) \sigma_{a, g} (-\bi)^{|a|} v_g^\alpha
= \frac{1}{2} \sum_{\alpha\in\widehat{\ZZ}_2} \alpha(a) \sigma_{a, g} (-\bi)^{|a|}
\sum_{b\in\ZZ_2} \alpha(b) \sigma_{b, g} \bi^{|b|}x_g^b \\
&= \sum_{b\in\ZZ_2} \sigma_{a, g} \sigma_{b, g} (-\bi)^{|a|} \bi^{|b|} \left[\frac{1}{2}
\sum_{\alpha\in\widehat{\ZZ}_2} \alpha(a + b) \right] x_g^b \\
&= \sum_{b\in\ZZ_2} \sigma_{a, g} \sigma_{b, g} (-\bi)^{|a|} \bi^{|b|} \delta_{a, b} x_g^b  = x_g^a,
\end{aligned} \end{equation*}
where we have used \eqref{bi.equation}.

Also, it is easy to see from \eqref{cd.from.cartan} that the basis $\{v_g^\alpha \med g\in\ZZ_2^{m-1}, \alpha\in\widehat{\ZZ}_2\}$
is a homogeneous basis for the Cartan grading because when $m = 3$ we have that:
\begin{equation} \label{cartan.bases.relation} \begin{aligned}
v^\mathds{1}_0 &= e_1 / \sqrt{2}, \;& v^\omega_0 &= e_2 / \sqrt{2}, \;&
v^\mathds{1}_{a_2} &= u_1 / \sqrt{2}, \;& v^\omega_{a_2} &= v_1 / \sqrt{2}, \\
v^\mathds{1}_{a_3} &= u_2 / \sqrt{2}, \;& v^\omega_{a_3} &= v_2 / \sqrt{2}, \;&
v^\mathds{1}_{a_2 +a_3} &= -u_3 / \sqrt{2}, \;& v^\omega_{a_2+a_3} &= -v_3 / \sqrt{2},
\end{aligned} \end{equation}
where $\widehat{\ZZ}_2 = \langle\omega\rangle$.

Furthermore, we claim that
\begin{equation} \label{product.cd.by.cartan}
x_g^h v_k^\alpha = 
(\alpha \omega^{|k|} \sigma_k)(h) \sigma_{(g,h),(k,h)} \bi^{|h|}
v_{g+k}^{\alpha \omega^{|g|}} \in \FF v_{g+k}^{\alpha \omega^{|g|}}
\end{equation}
for each $g,k \in\ZZ_2^{m-1}$, $h\in\ZZ_2$, $\alpha\in\widehat{\ZZ}_2$, where $\widehat{\ZZ}_2 = \langle\omega\rangle$.
Indeed, we can first check that:
\begin{equation*} \begin{aligned}
x_g^h v_k^\alpha &= x_g^h \frac{\sqrt{2}}{4} \sum_{a\in\ZZ_2} \alpha(a) \sigma_{a,k} \bi^{|a|} x_k^a
= \frac{\sqrt{2}}{4} \sum_{a\in\ZZ_2} \alpha(a) \sigma_{a,k} \sigma_{(g,h), (k,a)} \bi^{|a|} x_{g+k}^{h+a} \\
&= \frac{\sqrt{2}}{4} \sum_{a\in\ZZ_2} \alpha(a+h) \sigma_{a+h,k} \sigma_{(g,h), (k,a+h)} \bi^{|a+h|} x_{g+k}^a \\
&= \alpha(h) \sigma_{(g,h),(k,h)} \bi^{|h|} \frac{\sqrt{2}}{4} \sum_{a\in\ZZ_2}
\big( \alpha \sigma_{(g,h)} \sigma_h \big)(a) \cdot \\
&\qquad \cdot \Big[\sigma_{k,a+h}(-1)^{|a+h||k|} \Big]
\Big[\sigma_{a,g+k}\sigma_{g+k,a}(-1)^{|a||g+k|} \Big] \bi^{|a|} x_{g+k}^a \\
&= \alpha(h) \sigma_{(g,h),(k,h)} \sigma_{k,h} \bi^{|h|} \frac{\sqrt{2}}{4} \sum_{a\in\ZZ_2}
\big( \alpha \sigma_{(g,h)}\sigma_{g+k} \sigma_h \sigma_k \big)(a) \cdot \\
& \qquad \cdot \sigma_{a,g+k} \omega^{|g+k|}(a) \omega^{|k|}(a+h) \bi^{|a|} x_{g+k}^a \\
&= (\alpha \omega^{|k|} \sigma_k)(h) \sigma_{(g,h),(k,h)} \bi^{|h|}
v_{g+k}^{\alpha \sigma_{(g,h)} \sigma_{g+k} \sigma_h \sigma_k \omega^{|g+k|+|k|}}
\end{aligned} \end{equation*}
where we have used \eqref{bi.equation}, \eqref{sigma.other.equation} and the property $(-1)^{|a|} = \omega(a)$ for $a\in\ZZ_2$.
If we denote $\psi \coloneqq \alpha \sigma_{(g,h)} \sigma_{g+k} \sigma_h \sigma_k \omega^{|g+k|+|k|}$,
we need to prove that $\psi = \alpha \omega^{|g|}$. The second column of $\sigma$ in \eqref{sigmamatrix} shows that
$\sigma_{(g,h)}(a) = (\sigma_g\sigma_h)(a)$ for $a\in\ZZ_2 = \langle a_1 \rangle$, so that $\psi = \alpha \sigma_g \sigma_{g+k} \sigma_k \omega^{|g+k|+|k|}$.
It is clear that \eqref{product.cd.by.cartan} follows from the equalities above in the case $g=0$, so we can assume from now on that $g\neq 0$.
If $k\in\{0,g\}$, then $\sigma_g \sigma_{g+k} \sigma_k = \mathds{1}$ and $\omega^{|g+k|+|k|} = \omega = \omega^{|g|}$; 
on the other hand, if $k\notin\{0,g\}$, we have that $\omega^{|g+k|+|k|} = \omega^2 = \mathds{1}$ and $\psi = \alpha\sigma_{a_2}\sigma_{a_3}\sigma_{a_2+a_3}=\alpha\omega$; in all cases \eqref{product.cd.by.cartan} holds, which proves the claim.
\end{notation}

\begin{remark}
In further sections we will see that, on the Kantor pairs of Hurwitz type, the pairs $(v^{\alpha}_g, v^{\omega\alpha}_g)$ are homogeneous idempotents for a grading that is related to the Cartan grading on the associated Kantor-Lie algebra. Furthermore, we will use the fact that \eqref{product.cd.by.cartan} describes certain automorphisms of the associated Kantor triple systems that permute the subspaces $\FF v_g^{\alpha}$.
\end{remark}

\subsection{Kantor systems of Hurwitz type}
It may be convenient to introduce the following terminology:
\begin{df}
Kantor pairs and triple systems associated with a Hurwitz algebra $C$ will be referred to as \emph{Hurwitz pairs} and \emph{Hurwitz triples}, and denoted by $\cV_C$ and $\cT_C$, respectively; the term \emph{Hurwitz (Kantor)  system} will be used to refer to any of them. For each possible case for $\dim C = 1,2,4,8$, the corresponding Hurwitz systems will be called Hurwitz systems of types \emph{unarion}, \emph{binarion}, \emph{quaternion} and \emph{octonion}, respectively.
\end{df}

\begin{remark}
A straightforward calculation shows that the $U$-operator for a Kantor system associated with a Hurwitz algebra $C$ is given by
\begin{equation}
U_x(y) \coloneqq \{x, y, x\} = 2 n(x,y)x -3 n(x)y \label{eq_Ux_Hurwitz} \\
\end{equation}
for any $x,y\in C$.
\end{remark}

\begin{proposition}
Let $C$ be a Hurwitz algebra of dimension greater than $1$. Then, the main grading on the Kantor-Lie algebra $\kan(C) := \kan(\cV_C)$ produces a Jordan pair given by $\big(\kan(C)^2, \kan(C)^{-2} \big)$ that is isomorphic to a simple Jordan pair of type ${\rm IV}_k$ with $k = \dim \cS(C)$.
\end{proposition}
\begin{proof}
Denote $\cV = \big(\kan(C)^2, \kan(C)^{-2} \big)$ and $\cS = \cS(C)$. Let $\cW_k = (\FF^k, \FF^k)$ denote the simple Jordan pair of type $\text{IV}_k$ (see \cite[Chapter 4]{L75}); recall that the triple product of $\cW_k$ is determined by $Q_x(y) := \frac{1}{2} \{x,y,x\} = q(x,y)x - q(x)y$, where $q$ is the quadratic form associated with the standard scalar product of $\FF^k$. By restriction of the main grading on $\kan(C)$, it is clear that $\cV$ generates a Kantor-Lie algebra whose main grading is a $3$-grading, which forces $\cV$ to be a Jordan pair.

By definition of the Kantor construction, we have that $\cV^\sigma = \{\iota^\sigma(L_x) \med x\in \cS\}$ for $\sigma = \pm$, where
$\iota^\sigma \colon \cV^\sigma \to \kan(C)^{\sigma 2} \subseteq \kan(C)$
are the inclusions given by
\begin{equation} \label{inclusions.kantor}
\iota^+(f) := \left( \begin{matrix} 0 & 0 \\ f & 0 \end{matrix} \right), \quad
\iota^-(f) := \left( \begin{matrix} 0 & f \\ 0 & 0 \end{matrix} \right).
\end{equation}
A straightforward computation shows that
$$
\{\iota^\sigma(L_x), \iota^{-\sigma}(L_y), \iota^\sigma(L_z)\} = 
\big[[\iota^\sigma(L_x), \iota^{-\sigma}(L_y)], \iota^\sigma(L_z)\big] =
\iota^\sigma(L_x L_y L_z + L_z L_y L_x).
$$
By the left Moufang identity (\cite[Chapter 2]{ZSSS82}) we know that $L_x L_y L_x = L_{xyx}$. 
Since $\bar x = -x$, from the identities $\bar a b + \bar b a = n(a, b)$, $(a b) b = a b^2 $ and $\bar a a = n(a)$,
we get that $xyx = - (\bar x y) x =  - n(x,y)x + y (\bar x x) = n(x)y - n(x,y)x$. Thus
$$
\frac{1}{2} \{\iota^\sigma(L_x), \iota^{-\sigma}(L_y), \iota^\sigma(L_x)\} =
\iota^\sigma(L_x L_y L_x) = \iota^\sigma(L_{xyx}) = \iota^\sigma(L_{n(x)y - n(x,y)x}).
$$
Let $\bi$ be a square root of $-1$ in $\FF$. Then the map 
$$\cW_k^\sigma \equiv \cS \longrightarrow \cV^\sigma, \quad x \longmapsto \bi \iota^{\sigma}(L_x)$$
defines an isomorphism, where the quadratic form on $\cS$ is just $n{\big|}_\cS$.
\end{proof}

\subsection{Peirce and root space decompositions}

In this section, we will give a result that, under certain restrictions, gives a relation between a Peirce decomposition on a Kantor pair and the root system of the associated Kantor-Lie algebra.

\begin{df}
An element $e=(e^+,e^-) \in \cV=(\cV^+,\cV^-)$ is called an {\em idempotent} of the Kantor pair $\cV$ if $\lbrace e^{\sigma}, e^{-\sigma}, e^{\sigma} \rbrace = e^{\sigma}$ for each $\sigma\in \{ +,- \}$. Define the linear operators $L^{\sigma}(e), R^{\sigma}(e) \in \End(\cV^{\sigma})$ as $L^{\sigma}(e)x^{\sigma} := \{ e^{\sigma}, e^{-\sigma}, x^{\sigma} \}$ and $R^{\sigma}(e)x^{\sigma} := \{x^{\sigma}, e^{-\sigma}, e^{\sigma} \}$ for $x^{\sigma} \in \cV^{\sigma}$ and $\sigma\in \{+,- \}$. 

Similarly, an element $e$ of a Kantor triple system $\cT$ is called a \emph{tripotent} if $\{e,e,e\} = e$, and we define the operators $L(e), R(e) \in \End(\cT)$ as $L(e)x := \{ e, e, x \}$ and $R(e)x := \{x, e, e\}$ for $x \in \cT$.
\end{df}

It is well-known \cite{KK03} that, if $\chr\FF\neq 2,3,5$, a tripotent $e$ of a Kantor triple system $\cT$ produces a \emph{Peirce decomposition} given by
\begin{equation}
\cT= \cT_{0,0} \oplus \cT_{\frac{1}{2},\frac{1}{2}} \oplus \cT_{1,1} \oplus \cT_{\frac{3}{2},\frac{3}{2}} \oplus \cT_{-\frac{1}{2},0} \oplus \cT_{0,1} \oplus \cT_{\frac{1}{2},2} \oplus \cT_{1,3}
\end{equation}
where 
\begin{equation}
\cT_{\lambda, \mu} = \cT_{\lambda, \mu}(e) := \{x\in\cT \med L(e)x = \lambda x, R(e)x = \mu x\}.
\end{equation}
For the Peirce subspace $\cT_{\lambda, \mu}$, the scalars $\lambda$ and $\mu$ are called its associated \emph{Peirce constants}. Peirce decompositions are important in theory of Jordan pairs and Jordan triple systems (see for instance \cite{MC03}), and have also been studied in the case of structurable algebras, but it is unknown to the authors if these have appeared in the literature in the context of non-Jordan Kantor pairs. Note that in the case of Kantor pairs, the Peirce subspaces associated with an idempotent can be defined by:
\begin{equation}
\cV_{\lambda, \mu}^{\sigma} = \cV_{\lambda, \mu}^{\sigma}(e) := \{x\in\cV^{\sigma} \med L^{\sigma}(e)x = \lambda x, R^{\sigma}(e)x = \mu x\}.
\end{equation}
Note that it is unknown if an idempotent of a Kantor pair will always produce a Peirce decomposition under our general assumptions on $\chr\FF$.

\begin{df}
Let $L$ be a finite-dimensional semisimple Lie algebra over an algebraically closed field $\FF$ of characteristic $0$, with Killing form $\kappa$. Let $H$ be a fixed Cartan subalgebra of $L$. We have the well-known {\em root space decomposition}:
\[
L=H \oplus \bigoplus_{\alpha \in \Phi} L_{\alpha},
\] 
where $L_{\alpha}=\lbrace x\in L \mid [h,x]=\alpha(h)x \mbox{ for all } h\in H\rbrace$. The elements in $\Phi \subseteq H^* \setminus \lbrace 0 \rbrace$ are called {\em roots} and $\Phi=\Phi^+ \cup \Phi^-$ where $\Phi^+$ (resp. $\Phi^-$) is the set of positive (resp. negative) roots.  
Let $\Delta$ be a system of simple roots, which generates $\Phi$. (See \cite[\S3]{EKmon}).
\end{df}

Let $\Delta= \lbrace \alpha_1 ,..., \alpha_n \rbrace$ for $n\in \NN$. We can consider a map $\omega:\Delta\rightarrow \ZZ$. Since $\Delta$ generates $\Phi$, we get a linear operator $\omega: \Phi\rightarrow \ZZ$.  By defining $L_i:= \oplus_{\omega(\alpha)=i} L_{\alpha}$ for $\alpha\in \Phi$ we get a  $\ZZ$-grading on $L$, $L= \oplus_{i} L_{i}$. It is clear that this decomposition is a grading since $[L_{\alpha}, L_{\beta}]\subseteq L_{\alpha + \beta}$ for all $\alpha,\beta \in \Phi$. 

In particular, we can choose an $\omega: \Phi\rightarrow \ZZ$ such that $\omega(\Delta) \subseteq \{0,1\}$ and that for all $\alpha \in \Phi^+$ we get $\omega(\alpha)\leq 2$, analogously for $\alpha \in \Phi^-$; then we get a $5$-grading. Recall from the Kantor construction that $\cV = (L_1, L_{-1})$ defines a Kantor pair.

For any $\alpha \in H^*$, let $t_{\alpha}\in H$ be such that $\alpha(h)=\kappa(t_{\alpha},h)$ for all $h\in H$. Since the restriction of $\kappa$ to $H$ is nondegenerate, it induces a nondegenerate symmetric bilinear form $(\;\mid\;): H^*\times H^* \rightarrow \FF$ given by $(\alpha \mid \beta)= \kappa(t_{\alpha},t_{\beta})$.

\begin{remark}
Consider $L$, $H$, $H^*$ and $\Phi$ as above. We claim that for each $\alpha \in \Phi$, the Kantor pair defined by $(L_{\alpha}, L_{-\alpha})$ is isomorphic to $\cV_\FF$ (the $1$-dimensional simple Kantor pair), or equivalently, the Kantor pair $(L_{\alpha}, L_{-\alpha})$ is spanned by an idempotent $e = (e_{\alpha}, e_{-\alpha})$.

Indeed, it is well-known that each pair of roots $\lbrace \alpha, -\alpha \rbrace$ generates a Lie algebra isomorphic to $\mathfrak{sl}_2(\FF)$; hence there exist $x_{\pm \alpha}\in L_{\pm \alpha}$ such that $[x_{\alpha},x_{-\alpha}]=h_{\alpha}$, $[h_{\alpha},x_{\alpha}]=2 x_{\alpha}$ and $[h_{\alpha},x_{-\alpha}]=-2 x_{-\alpha}$ (see \cite[\S3]{EKmon} and references therein). Therefore, it is straightforward to see that $(x_{\alpha},\dfrac{1}{2} x_{-\alpha})$ is an idempotent of $(L_{\alpha}, L_{-\alpha})$. (Note that this also follows from the fact that $\mathfrak{sl}_2(\FF)$ is the Kantor-Lie algebra associated with the $1$-dimensional Kantor pair $(\FF, \FF)$.)

More generally, given an idempotent $e = (e_{\alpha}, e_{-\alpha})$ of $(L_{\alpha}, L_{-\alpha})$, it is easy to see that any idempotent of $(L_{\alpha}, L_{-\alpha})$ has the form $(\lambda e_{\alpha},\lambda^{-1} e_{-\alpha})$ for $\lambda\in\FF^\times$. Any of these idempotents will be called an {\em idempotent associated with the root $\alpha$}. Note that given an idempotent $e$ of $(L_{\alpha}, L_{-\alpha})$, the Peirce operators $L^{\sigma}(e)$ and $R^{\sigma}(e)$ do not depend on the choice of the idempotent, and this remains true for any Kantor pair containing $(L_{\alpha}, L_{-\alpha})$.
\end{remark}

Next result shows a relation between roots and left Peirce constants on Kantor pairs.

\begin{proposition} \label{peirce_constants}
Let $L$ be a finite-dimensional semisimple Lie algebra over an algebraically closed field $\FF$ of characteristic $0$. Let $H$ be a Cartan subalgebra of $L$ and $\Phi$ its associated root system. Let $\cV$ be a nonzero Kantor pair associated with a $\ZZ$-grading compatible with the Cartan grading on $L$, that is, $\cV = (L_1, L_{-1})$ for some $\ZZ$-grading that is a coarsening of the Cartan grading. Let $e^{\alpha} := (e_{\alpha},e_{-\alpha}) $ be an idempotent of $\cV$ associated with some $\alpha\in \Phi$. Then, for any $\beta \in \Phi$ such that $L_{\pm\beta}\subseteq\cV^{\pm}$, the idempotent $e^{\alpha}$ acts multiplicatively on $L_{\pm\beta}$ with left Peirce constant
$\lambda = \dfrac{(\alpha\mid\beta) }{(\alpha\mid\alpha)} =
\dfrac{\left\lVert \beta \right\rVert}{\left\lVert \alpha \right\rVert} \cos(\alpha,\beta)$.
\end{proposition}

\begin{proof}
Since $e^{\alpha}$ is an idempotent and $[x,y]=\kappa(x,y)t_{\alpha}$ for $x\in L_{\alpha}$ and $y\in L_{-\alpha}$ (\cite[\S8.3]{H78}), we have $e_{\alpha}=\lbrace e_{\alpha}, e_{-\alpha},e_{\alpha} \rbrace = [[e_{\alpha},e_{-\alpha}],e_{\alpha}]=\kappa(e_{\alpha},e_{-\alpha}) [t_{\alpha},e_{\alpha}]=\kappa(e_{\alpha},e_{-\alpha}) \alpha(t_{\alpha}) e_{\alpha}= \kappa(e_{\alpha},e_{-\alpha}) (\alpha\mid\alpha) e_{\alpha}$. Then $1= \kappa(e_{\alpha},e_{-\alpha}) (\alpha\mid\alpha) $.

Let $\beta \in \Phi$ and $x_{\beta}\in L_{\beta}$ such that $x_{\beta} \in \cV^{\sigma}$ for some $\sigma = \pm$. Then
$$
\{ e_{\alpha}, e_{-\alpha}, x_{\beta} \} = [[e_{\alpha},e_{-\alpha}],x_{\beta}]=\kappa(e_{\alpha},e_{-\alpha})[t_{\alpha}, x_{\beta}]=\dfrac{\beta(t_{\alpha})}{(\alpha\mid \alpha)} x_{\beta} =\dfrac{(\alpha\mid\beta)}{(\alpha\mid\alpha)} x_{\beta},
$$
with $\dfrac{(\alpha\mid\beta) }{(\alpha\mid\alpha)} =
\dfrac{\left\lVert \alpha \right\rVert \left\lVert \beta \right\rVert \cos(\alpha, \beta)}{ \left\lVert \alpha \right\rVert^2} =
\dfrac{\left\lVert \beta\right\rVert}{\left\lVert \alpha \right\rVert} \cos(\alpha, \beta)$.
Similarly, for $x_{-\beta} \in L_{-\beta}$, we have that
$$
\{ e_{-\alpha}, e_{\alpha}, x_{-\beta} \} = \dfrac{(-\alpha \mid -\beta)}{(-\alpha \mid -\alpha)} x_{-\beta}
= \dfrac{(\alpha\mid\beta)}{(\alpha\mid\alpha)} x_{-\beta}.
$$
\end{proof}

\section{Generalities on Kantor systems} \label{section.generalities}
\subsection{On gradings and automorphisms}
We will now prove some general results on Kantor systems. Some of these extend analogous results for Jordan systems given in \cite[\S 2]{Ara17}.

\begin{proposition} \label{StabilizerIdentity}
Let $\cA$ be a structurable $\FF$-algebra with unity $1$ and denote by $\cT$ and $\cV$ the associated Kantor triple system and Kantor pair, respectively. Then the automorphism group scheme $\AAut(\cA, \inv)$ is the stabilizer of $1$ in $\AAut(\cT)$.

In addition, suppose that either $\chr\FF\neq3$, or that $\cA$ is the algebra generated by $\cH(\cA,\inv) \cup \cS(\cA,\inv)^2$. Then $\AAut(\cA, \inv)$ is the stabilizer of $(1^+, 1^-)$ in $\AAut(\cV)$.
\end{proposition}
\begin{proof}
Let $R$ denote a commutative associative unital $\FF$-algebra and consider the $R$-algebra $\cA_R = \cA \otimes R$ with the extended $R$-linear involution. It is clear that $\Aut_R(\cA_R, \inv) \leq \Aut_R(\cT_R) \leq \Aut_R(\cV_R)$. Denote $\cH = \cH(\cA_R,\inv)$ and $\cS = \cS(\cA_R,\inv)$.

First, we claim that the involution and product of $\cA_R$ are determined by $1^\sigma$ ($\sigma=\pm$) and the triple product. On the one hand, we can recover the involution using that $\bar x = 2x - \{x,1,1\}$, which determines the subspaces $\cH$ and $\cS$. Let $h\in\cH$ and $z\in\cA_R$; then, $hz = \{h,1,z\}$ and $zh = \overline{h \bar z}$. Now, take $s,t\in\cS$; we have $st = -\{t,s,1\}$. Therefore, the product is recovered too, which proves the claim.

Let $\varphi=(\varphi^+,\varphi^-) \in\Aut_R(\cV_R)$ with $\varphi^\sigma(1^\sigma)=1^\sigma$ for $\sigma=\pm$. Notice that the involution of $\cA_R$ commutes with $\varphi^\sigma$ because 
$$ \varphi^\sigma(\bar x) = \varphi^\sigma(2x - \{x,1,1\}) 
= 2\varphi^\sigma(x) - \{\varphi^\sigma(x), 1 , 1\} = \overline{\varphi^\sigma(x)},$$
and therefore the subspaces $\cH$ and $\cS$ are $\varphi^\sigma$-invariant. Note that $U_1$ has eigenspaces $\cH$ and $\cS$ with associated eigenvalues $1$ and $-3$, respectively. Take $h\in\cH$; then $\varphi^\sigma(h)\in\cH$ and $\varphi^+(h) = \varphi^+(U_1(h)) = U_1(\varphi^-(h)) = \varphi^-(h)$, from where it follows that the maps $\varphi^+$ and $\varphi^-$ coincide in $\cH$. 

Fix $s\in\cS$. In the case that $\chr\FF\neq3$ we have that $\varphi^+(s) = \varphi^+(-\frac{1}{3}U_1(s)) = -\frac{1}{3}U_1(\varphi^-(s)) = \varphi^-(s)$, thus $\varphi^+$ and $\varphi^-$ coincide in $\cS$ too, and therefore $\varphi^+=\varphi^-$.

Now, consider the case where $\cA$ is the $\FF$-algebra generated by $\cH(\cA,\inv) \cup \cS(\cA,\inv)^2$. Then $\cA_R$ is the $R$-algebra generated by $\cH \cup \cS^2$. Since we have that $\{t,s,1\} = -st$ for each $s,t\in\cS$, it follows that
$$\varphi^\sigma(st) = \varphi^\sigma(-\{t,s,1\}) = -\{\varphi^\sigma(t), \varphi^{-\sigma}(s), 1\} 
= \varphi^{-\sigma}(s)\varphi^\sigma(t)$$
and
$$\varphi^\sigma(st) = \varphi^\sigma(\overline{ts}) = \overline{\varphi^\sigma(ts)} 
= \overline{\varphi^{-\sigma}(t)\varphi^\sigma(s)} = \varphi^\sigma(s) \varphi^{-\sigma}(t).$$ 
In consequence, $\varphi^+(st) = \varphi^+(s) \varphi^-(t) = \varphi^-(st)$, that is, $\varphi^+$ and $\varphi^-$ coincide in $\cS^2$. Let $x_1,x_2,x_3\in\cA_R$ be such that $\varphi^+(x_i) = \varphi^-(x_i)$ and set $z=\{x_1,x_2,x_3\}$. Then $\varphi^+(z) = \varphi^-(z)$. Indeed, $\varphi^+(z) = \varphi^+(\{x_1,x_2,x_3\}) = \{\varphi^+(x_1),\varphi^-(x_2),\varphi^+(x_3)\} 
= \{\varphi^-(x_1),\varphi^+(x_2),\varphi^-(x_3)\} = \varphi^-(\{x_1,x_2,x_3\}) = \varphi^-(z)$.
Since the product of $\cA_R$ can be recovered from the triple product and the unity, an inductive argument shows that $\varphi^+$ and $\varphi^-$ coincide in the algebra generated by $\cH \cup \cS^2$, which is $\cA_R$. Thus we get again $\varphi^+ = \varphi^-$.

We need to prove that $\varphi\coloneqq\varphi^+\in\Aut_R(\cA_R, \inv)$. If $h\in\cH$ and $z\in\cA_R$, we have $\varphi(h)\in\cH$ and $hz = \{h,1,z\}$, and so $\varphi(hz) = \varphi(\{h,1,z\}) = \{\varphi(h),1,\varphi(z)\} = \varphi(h)\varphi(z)$. Also, since $zh=\overline{h\overline{z}}$, we get that $\varphi (zh)=\varphi (z) \varphi (h)$.
If $s,t\in\cS$, we have that $st = -\{t,s,1\}$ and again it follows easily that $\varphi(st)=\varphi(s)\varphi(t)$. Therefore $\varphi\in\Aut_R(\cA_R, \inv)$. This proves the second statement of the result. The first statement can be proven with similar arguments.
\end{proof}

The following example shows that the conditions in Proposition~\ref{StabilizerIdentity} cannot be dropped:

\begin{example}
Let $\FF$ be a field of characteristic $3$ and consider the $2$-dimensional split Hurwitz algebra $K=\FF\times\FF$ with involution $(x,y)\mapsto(y,x)$. (Note that the subalgebra generated by $\cH(K,\inv)\cup \cS(K,\inv)^2$ is $\FF1$, which does not equal $K$.) The element $s=(1,-1)\in K$ is skew-symmetric and satisfies $s^2=1$. Let $\cV = (K, K)$ denote the Kantor pair associated with $K$ and fix $\lambda\in\FF^\times$. We claim that the maps
\begin{equation}\begin{aligned}
\varphi^\sigma \colon & K \longrightarrow K, \\
1 & \longmapsto 1, \; s\longmapsto \lambda^{\sigma1} s,
\end{aligned}\end{equation}
for $\sigma\in\{+,-\}$, define an automorphism $\varphi$ of $\cV$. Indeed, this is straightforward because the triple product of $K$ is determined by
\begin{align*}
& \{1,1,1\} = -\{s,s,1\} = 1, \qquad \{1,1,s\} = -\{s,s,s\} = s, \\
& \{1,s,1\} = \{s,1,1\} = \{1,s,s\} = \{s,1,s\} = 0.
\end{align*}
However $\varphi^+\neq\varphi^-$ and $\varphi^\sigma\notin\Aut(K)$ unless $\lambda\in\{\pm 1\}$.
\end{example}

\begin{proposition} \label{finegradingproperty} 
Let $\Gamma$ be a fine grading on a Kantor pair $\cV$ and let $G$ be its universal group. Then, there is a group homomorphism $\pi \colon G \rightarrow \ZZ$ such that $\pi(g)=\sigma1$ if $\cV^\sigma_g\neq0$ for some $\sigma\in\{+,-\}$. In particular, $\supp\Gamma^+$ and $\supp\Gamma^-$ are disjoint.
\end{proposition}
\begin{proof}
This follows from the same arguments of the proof for the particular case of Jordan pairs, given in \cite[Prop.~2.2]{Ara17}.
\end{proof}

\begin{proposition}\label{regulargradspairs} 
Let $\cA$ be a structurable $\FF$-algebra with unity $1$, and $G$ an abelian group. In case that $\chr\FF = 3$, assume also that $Z_c(\cA)\cap\cS(\cA, \inv) = 0$, where $Z_c(\cA):= \lbrace z\in \cA \med zx=xz \mbox{ for any } x\in \cA \rbrace$ is the commutative center of $\cA$,
 and that $\cA$ is generated as an algebra by $\cH(\cA,\inv) \cup \cS(\cA,\inv)^2$. Consider the associated Kantor pair $\cV=(\cA, \cA)$.

If $\Gamma$ is a $G$-grading on $\cV$ such that $1^+$ (or $1^-$) is homogeneous, then the restriction of the shift $\Gamma^{[g]}$ to $\cA=\cV^+$, with $g=-\deg(1^+)$, defines a $G$-grading $\Gamma_\cA$ on $\cA$.
Moreover, if $G=\Univ(\Gamma)$ then the universal group $\Univ(\Gamma_\cA)$ is isomorphic to the subgroup of $\Univ(\Gamma)$ generated by  $\supp \Gamma^{[g]}$; if in addition $\Gamma$ is fine we also have that $\Univ(\Gamma) \cong \Univ(\Gamma_\cA) \times \ZZ$, and the universal degree of $\Gamma$ is equivalent to $\deg_\Gamma(x^\sigma) := (\deg_{\cA}(x), \sigma 1)$ where $\deg_{\cA}$ is the universal degree of $\Gamma_\cA$.
\end{proposition}
\begin{proof}
Denote $\cH := \cH(\cA, \inv)$ and $\cS := \cS(\cA, \inv)$. Let $\Gamma$ be a $G$-grading on $\cV$ where $1^+$ is homogeneous. We claim that $1^-$ is homogeneous.

First, consider the case where $\chr\FF \neq 3$. Then $U_1^+(h^-)=h^+$ and $U_1^+(s^-)=-3s^+$ for $h\in\cH$, $s\in\cS$.
Since the homogeneous map $U_1^+$ is invertible, there exists a homogeneous element $y^-\in\cA^-$ such that $U_1^+(y^-)=1^+$, and since this is only possible for $y^- = 1^-$, we get that $1^-$ is homogeneous, which proves the claim in this case.

Consider now the case with $\chr\FF = 3$. Note that $V_{1,s}(x) = [x,s]$ for any $x\in\cA$, $s\in\cS$; since $Z_c(\cA)\cap\cS(\cA, \inv) = 0$, it follows that $V_{1,s} = 0$ if and only if $s = 0$; in other words, the linear map $\cS \to \End(\cV^+)$, $s \mapsto V_{1,s}$ is injective. The subspaces $\im U_1^+ = \cH^+$ and $\ker U_1^+ = \cS^-$ are graded. Since $1^+ \in\im U_1^+$, there exists a homogeneous element $y^-$ such that $U_1^+(y^-) = 1^+$, and we have that $\deg(1^+) + \deg(y^-) = 0$. It is clear that $y^- - 1^- \in \ker U_1^+ = \cS^-$, so that $y^- = 1^- + s^-$ for some $s\in\cS$. If $V_{1,s} = 0$, then $s = 0$ and $1^-$ is homogeneous. Assume now that $V_{1,s} \neq 0$. Then $0 \neq V_{1,s} = V_{1,y} - V_{1,1} = V_{1,y} - \id$ is a homogeneous map of degree 0. Since $\cS^-$ is graded, we can write $s = \sum_i s_i$ where $s_i^-$ are homogeneous and have different degrees in $\cS^-$. Consequently, we have that $V_{1,s} = \sum_i V_{1,s_i}$ with $V_{1,s_i} \neq 0$ for each $i$, and since $V_{1,s}$ has degree $0$, it follows that $s = s_i$ for some $i$, that is, $s$ is homogeneous in $\cV^-$. Moreover, $\deg(s^-) = -\deg(1^+) = \deg(y^-)$, and $1^-$ is homogeneous.

We have proven that $1^-$ is homogeneous in every case. If $\deg_g$ is the degree map of the grading $\Gamma^{[g]}$ on $\cV$, we have that $\deg_g(1^+) = 0 = \deg_g(1^-)$ because $U_1^+(1^-) = 1^+$. By Proposition~\ref{StabilizerIdentity}, the stabilizer of $(1^+, 1^-)$ in $\AAut(\cV)$ is $\AAut(\cA, \inv)$. Note that the ``automorphisms'' in $\AAut(\cV)$ producing the grading must be in the stabilizer of $(1^+, 1^-)$, which is $\AAut(\cA, \inv)$. Therefore, using the correspondence between gradings and morphisms of affine group schemes, it follows that the grading $\Gamma^{[g]}$ on $\cV$ is produced by a morphism $\Hom_{\text{Alg}_{\FF}}(\FF G, -) \to \AAut(\cA, \inv)$, which also defines a grading $\Gamma_{\cA}$ on $\cA$. It also follows that the homogeneous components of $\Gamma^{[g]}$ coincide in $\cV^+$ and $\cV^-$.

\smallskip

Assume now that $G = \Univ(\Gamma)$ and call $H = \langle \supp \Gamma^{[g]} \rangle$. Note that $\Gamma^{[g]}$ can be regarded as a $\Univ(\Gamma)$-grading and also as an $H$-grading; similarly $\Gamma_\cA$ can be regarded as a $\Univ(\Gamma_\cA)$-grading and as an $H$-grading. By the universal property of the universal group, the $H$-grading $\Gamma_\cA$ is induced from the $\Univ(\Gamma_\cA)$-grading $\Gamma_\cA$ by an homomorphism $\varphi_1 \colon \Univ(\Gamma_\cA) \to H$ that restricts to the identity in the support. On the other hand, the $\Univ(\Gamma_\cA)$-grading $\Gamma_\cA$ induces a $\Univ(\Gamma_\cA)$-grading $(\Gamma_\cA, \Gamma_\cA)$ on $\cV$ that is a coarsening of $\Gamma$, and therefore $(\Gamma_\cA,\Gamma_\cA)$ is induced from $\Gamma$ by some epimorphism $\varphi \colon \Univ(\Gamma) \to \Univ(\Gamma_\cA)$. Let $\varphi_2 \colon H \to \Univ(\Gamma_\cA)$ be the restriction of $\varphi$ to $H$. Note that $g\in\ker\varphi$, which implies that the $\Univ(\Gamma_\cA)$-grading $(\Gamma_\cA, \Gamma_\cA)$ is induced from the $H$-grading $\Gamma^{[g]}$ by $\varphi_2$, and also that $\varphi_2$ is an epimorphism which is the identity in the support. Since each epimorphism $\varphi_i$ is the identity in the support, both compositions $\varphi_1\varphi_2$ and $\varphi_2\varphi_1$ must be the identity and therefore $\Univ(\Gamma_\cA) \cong H$.

Suppose now that $\Gamma$ is fine and denote by $\Gamma_H$ the grading $\Gamma^{[g]}$ regarded as an $H$-grading. Note that $\Univ(\Gamma) = \langle \supp\Gamma^{[g]}, g \rangle = \langle \supp\Gamma_H, g \rangle = \langle H, g \rangle$. Consider $H$ as a subgroup of $H \times \ZZ \cong H \times \langle g_0 \rangle$, where the element $g_0$ has infinite order. The $H$-grading $\Gamma_H$ can be regarded as an $H\times \langle g_0 \rangle$-grading, and the shift $(\Gamma_H)^{[g_0]}$ defines another $H\times \langle g_0 \rangle$-grading where $\deg(1^+) = g_0$. Since the $H\times \langle g_0 \rangle$-grading $(\Gamma_H)^{[g_0]}$ is a coarsening of the $\Univ(\Gamma)$-grading $\Gamma$ (because $\Gamma$ is fine and by Proposition~\ref{finegradingproperty}), by the universal property there is an epimorphism $\Univ(\Gamma) = \langle H, g \rangle \to H\times \langle g_0 \rangle$ that sends $-g \mapsto g_0$ and fixes the elements of $H$. Consequently, $H \cap \langle g \rangle = 0$, $\langle g \rangle \cong \ZZ$, and we can conclude that $\Univ(\Gamma) = \langle H, g \rangle \cong H \times \ZZ \cong \Univ(\Gamma_\cA) \times \ZZ$. The statement about the universal degree follows too.
\end{proof}

\begin{proposition}\label{regulargradstriples} 
Let $\cA$ be a structurable $\FF$-algebra with unity $1$, and $\cT$ its associated Kantor triple system. If $\Gamma$ is a $G$-grading on $\cT$ such that $1$ is homogeneous, then the shift $\Gamma^{[g]}$ with $g=\deg(1)$ induces a $G$-grading $\Gamma_\cA$ on $\cA$. Moreover, if $G=\Univ(\Gamma)$ then $\deg(1)$ has order $2$ and $\Univ(\Gamma_\cA)$ is isomorphic to the subgroup of $\Univ(\Gamma)$ generated by $\supp \Gamma^{[g]}$; and if in addition $\Gamma$ is fine we also have that  $\Univ(\Gamma) \cong \Univ(\Gamma_\cA) \times \ZZ_2$, and the universal degree of $\Gamma$ is equivalent to $\deg_\Gamma(x) := (\deg_{\cA}(x), \bar 1)$ where $\deg_{\cA}$ is the universal degree of $\Gamma_\cA$.
\end{proposition}
\begin{proof} 
Let $\Gamma$ be a $G$-grading on $\cT$ such that $1$ is homogeneous. Then $\Gamma_\cV := (\Gamma, \Gamma)$ is a $G$-grading on the Kantor pair $\cV := \cV_\cA$ such that $1^+$ and $1^-$ are homogeneous. By Proposition~\ref{StabilizerIdentity}, and using the same argument from the proof of Proposition~\ref{regulargradspairs}, it follows that the restriction of the shift $\Gamma_\cV^{[g]}$ to $\cA=\cV^+$, with $g=-\deg(1^+)$, defines a $G$-grading $\Gamma_\cA$ on $\cA$. Since $U_1(1) = 1$, we get that $2g = 0$, so that $\Gamma^{[g]}$ defines a grading on $\cT$, and also $\Gamma_\cV^{[g]} = (\Gamma^{[g]}, \Gamma^{[g]})$. Thus $\Gamma^{[g]}$ defines the same $G$-grading $\Gamma_\cA$ on $\cA$, which proves the first claim.

Assume now that we also have $G=\Univ(\Gamma)$. We have shown above that $2\deg(1) = 0$. The $\ZZ_2$-grading $\Gamma_{\ZZ_2}$ given by $\cT = \cT_{\bar 1}$ is a coarsening of $\Gamma$, which implies by the universal property that there is a group epimorphism $\varphi \colon G \to \ZZ_2$ such that $\Gamma_{\ZZ_2}$ is induced by $\varphi$ from $\Gamma$. Consequently, $\deg(1)$ has order $2$. Set $H = \langle \supp\Gamma^{[g]}\rangle$. The rest of the proof follows with the same arguments from the last part of the proof of Proposition~\ref{regulargradspairs}, but using $\cT$ instead of $\cV_\cA$.
\end{proof}

\section{Automorphisms and orbits} \label{section.automorphisms.and.orbits}

In this section, the automorphism groups and their orbits are studied for Kantor systems of Hurwitz type.

\begin{notation} \label{notation.second.triple}
Let $\cV_C$ and $\cT_C$ be the Kantor pair and Kantor triple system, respectively, associated with a Hurwitz algebra $C$. For any $\lambda\in\FF^\times$, consider the pair of maps $c_\lambda := (c_\lambda^+, c_\lambda^-)$ defined by
\begin{equation}
c_\lambda^+(x) \coloneqq \lambda x, \qquad c_\lambda^-(y) \coloneqq \lambda^{-1} y,
\end{equation}
for any $x\in\cV_C^+$, $y\in\cV_C^-$. It is easy to see that $c_\lambda$ is an automorphism of $\cV_C$. Note that the $1$-torus $\langle c_\lambda \med \lambda\in\FF^\times\rangle$ produces the $\ZZ$-grading that is associated with the Kantor construction. Also, it is clear that $c_\lambda$ is an automorphism of $\cT_C$ if and only if $\lambda = \pm 1$.

We will denote by $L_a$ and $R_a$, respectively, the left and right multiplications by $a\in C$. We will also consider the triple system $\cT'_C \coloneqq C$ with the triple product given by $\{x,y,z\}' \coloneqq (x\bar y)z$.
\end{notation}

\begin{remark} \label{remark.orthogonal}
Let $V$ be a finite-dimensional vector space and $q \colon V\rightarrow \FF$ a nondegenerate quadratic form, and denote by $\Ort'(V,q)$ the reduced orthogonal group. Recall from \cite[4.8]{J89}) that $\Ort'(V,q) \trianglelefteq \Ort^+(V,q)$ and $\Ort^+(V,q)/\Ort'(V,q) \cong \FF^\times/(\FF^\times)^2$, where $(\FF^\times)^2$ is the multiplicative group of squares of $\FF^\times$. Since we assume that the base field is algebraically closed, we have that $\Ort'(V,q) = \Ort^+(V,q)$.

On the other hand, recall from \cite[Section 1]{Eld00} that if $\cC$ is a Cayley algebra with norm $n$, then
$\Ort'(\cC, n) = \langle L_a \med a\in \cC, n(a)=1 \rangle = \langle R_a \med a\in \cC, n(a)=1 \rangle$.
\end{remark}

\begin{lemma} \label{lemma.automorphisms}
Let $C$ be a Hurwitz algebra and $x\in C$ a traceless element with $n(x)=1$. Then $L_x \in\Aut(\cT_C)\cap\Aut(\cT'_C)$.
\end{lemma}
\begin{proof}
Recall from \cite[Section 3]{Eld96} that if $\cC$ is a Cayley algebra, we have that 
\begin{equation} \label{ternary.identity}
((xa)(\overline{xb}))(xc) = n(x)x((a\bar b)c)
\end{equation}
for all $x,a,b,c\in\cC$ where $x$ is traceless. Since any Hurwitz algebra is contained in a Cayley algebra, \eqref{ternary.identity} holds for any Hurwitz algebra, and consequently it is easy to see that $L_x$ behaves well with the triple products of $\cT_C$ and $\cT'_C$. (Note that $L_x^{-1} = L_{\bar x} = -L_x$.) 
\end{proof}

\bigskip

\begin{theorem}\label{th:schemestriple}
Let $C$ be a Hurwitz algebra, $\cT_C$ its associated Kantor triple system, and $\cT'_C$ the triple system defined in \ref{notation.second.triple}. Then $\AAut(\cT_C) = \AAut(\cT'_C) \leq \OO(C, n)$.
\end{theorem}
\begin{proof}
Denote $\cT = \cT_C$ and $\cT' = \cT'_C$. For any associative commutative unital $\FF$-algebra $R$, we will consider $\cT_R \coloneqq \cT \otimes R$ and $\cT'_R \coloneqq \cT' \otimes R$, with their $R$-linear triple products, and the $R$-algebra $C_R \coloneqq C\otimes R$. Let $f \in \Aut_R(\cT_R)$. Note that $\{x,x,x\}=n(x)x$ for any $x\in\cT_R$. Hence, $f(\{x,x,x\})=n(x)f(x)$ and $f(\{x,x,x\})=\{f(x),f(x),f(x)\}=n(f(x))f(x)$, so that $n(f(x)) = n(x)$ and $f\in O(C_R, n)$. It is clear that $\AAut(\cT') \leq \AAut(\cT)$. Finally, since we have that $\{x,y,z\} - \{z,y,x\} = -(z \bar x)y + (x \bar z)y = n(x,z)y -2(z \bar x)y = n(x,z)y - 2\{z,x,y\}'$ with $\AAut(\cT) \leq \OO(C, n)$, we get also that $\AAut(\cT) \leq \AAut(\cT')$.
\end{proof} 

\begin{corollary}\label{co:aut_triple}
Let $C$ be a Hurwitz algebra with norm $n$.\\
1) If $\dim C = 1$, then $\Aut(\cT_C) = O(C, n) \cong \ZZ_2$. \\
2) If $\dim C = 2$, then $\Aut(\cT_C) = O(C, n) \cong \FF^\times \rtimes \ZZ_2$. \\
3) If $\dim C = 4$, then $\Aut(\cT_C)\cong O^+(C, n)$. \\
4) If $\dim C = 8$, then $\Aut(\cT_C)\cong \Spin(\cS, n)$, where $\cS$ is the skew-symmetric subspace of $C$.
\end{corollary}
\begin{proof}
Cases 1), 2) and 3) are consequence of \cite[Corollary~6]{Eld96}, where they were proven for the triple system $\cT'_C$. Note that in case 2), the torus $\FF^\times$ corresponds to the automorphisms acting on the Cartan basis via 
\begin{equation} \label{torus.f.lambda}
f_\lambda(e_1) = \lambda e_1, \qquad f_\lambda(e_2) = \lambda^{-1} e_2.
\end{equation}
for $\lambda\in\FF^\times$, and the group $\ZZ_2$ corresponds to the swapping automorphism $e_1 \leftrightarrow e_2$. Finally, recall from \cite[Theorem~10]{Eld96} that we have an isomorphism $\Aut(\cT'_C)\cong \Spin(W, q)$, where $W$ is the set of traceless elements of $C$ with the quadratic form $q = -n$. It is clear that $W = \cS$. Finally, note that if $\ii$ is a square root of $-1$ in $\FF$, then the map $\cS \to \cS $, $s\mapsto \ii s$ extends to an isomorphism $\Spin(\cS, q) \to \Spin(\cS, n)$. 
\end{proof}

\begin{corollary}\label{co:aut.left}
Let $\cC$ be a Cayley algebra. Then, 
$$ \Aut(\cT_\cC) = \langle L_a \med a\in \cC, n(a)=1, t(a)=0 \rangle. $$
\end{corollary}
\begin{proof}
Set $G = \langle L_a \med a\in \cC, n(a)=1, t(a)=0 \rangle$. 
By Lemma~\ref{lemma.automorphisms}, we have that $G \leq \Aut(\cT_\cC)$. It is well-known that $\Aut(\cC, \inv) = \Aut(\cC)$. 
By \cite[Corollary 8]{Eld96} it follows that $\Aut(\cC) \leq G$.
Let $\varphi\in \Aut(\cT_\cC)$ and set $z = \varphi(1)$. Then, by the Cayley-Dickson doubling process, it follows that there exists $s\in\cC$ such that $n(s) = 1$, $t(s)=0$ and $t(sz) = 0$.
Then $L_s\in\Aut(\cT_\cC)$. Since $\Aut(\cT_\cC) \leq O(\cC, n)$, we have that $n(sz) = n(L_s\varphi(1)) = 1$, so that $L_{sz}\in\Aut(\cT_\cC)$. Besides, $L_{sz}L_s\varphi(1) = 1$. By Proposition~\ref{StabilizerIdentity}, we get that $L_{sz}L_s\varphi\in\Aut(\cC)\leq G$, which implies that $\varphi\in G$. This proves that $\Aut(\cT_\cC)\leq G$, and the result follows.
\end{proof} 

\begin{corollary} \label{co:equiv.grads.classif}
Let $C$ be a Hurwitz algebra. Then, the gradings and the equivalence classes of fine gradings are the same on $\cT_C$ and on $\cT'_C$.
\end{corollary}

\smallskip

\begin{remark} \label{remark.conjugate.inverse}
It is well known that if $\chr\FF \neq 3$, conjugate inverses of elements of simple structurable algebras are unique if they exist. Assume now that $\chr\FF = 3$ and let $C$ be a Hurwitz $\FF$-algebra.

Consider first the case where $C$ is the $2$-dimensional Hurwitz algebra $K$. Then, for any skew element $s$, we have that $V_{1,s} = 0$, but also $V_{1,1} = \id$, so that $V_{1,1 + s} = \id$. Therefore, the conjugate inverse of $1$ is not unique in this case.

Now consider the case $\dim C > 2$. Let $y$ be a conjugate inverse of $1$, that is $V_{1,y} = \id$. Set $y = \lambda 1 + s$ with $\lambda\in\FF$ and $s\in\cS(C,\inv)$. Then $1 = V_{1,y}(1) = V_{1, \lambda 1}(1) + V_{1,s}(1) = \lambda 1 + 0 = \lambda 1$, from where it follows that $\lambda = 1$. Thus $\id = V_{1,y} = V_{1,1} + V_{1,s} = \id + V_{1,s}$ and $V_{1,s} = 0$; but we also have that $V_{1,s}(x) = [x,s]$ for each $x\in C$. Hence, $[x, s] = 0$ for all $x\in C$, that is, $s\in Z_c(C)\cap\cS(C,\inv) = 0$ (where $Z_c(C)$ is the commutative center of $C$), so that $s = 0$. This proves that $1$ is the only conjugate inverse of $1$ in this case.
\end{remark}

\bigskip

We will now classify the orbits of Kantor pairs and triple systems associated with a Hurwitz algebra.

\begin{notation}
Let $C$ be a Hurwitz algebra and $\lambda\in\FF^\times$.
Let $n$ denote the (quadratic) norm of $C$ as a Hurwitz algebra.
(Note that in the $1$-dimensional case, the norm of a structurable algebra is linear, and so does not coincide with $n$.) \\
$\bullet$ If $\dim C > 1$, where we assume that either $\dim C \neq 2$ or $\chr\FF \neq 3$ (or both), we define:
\begin{align*}
\cO_0 &\coloneqq \{0\}, \qquad
\cO_1 \coloneqq \{0\neq x\in C \med n(x) = 0\}, \\
\cO_2 &\coloneqq \{x\in C \med n(x) \neq 0\}, \qquad
\cO_2(\lambda) \coloneqq \{x\in C \med n(x) = \lambda\}.
\end{align*}
$\bullet$ In the case $\dim C = 1$, we define:
\begin{align*}
\cO_0 \coloneqq \{0\}, \qquad \cO_1 \coloneqq \{x\in C\med x \neq 0\}, \qquad \cO_1(\lambda) \coloneqq \{x\in C\med n(x) = \lambda\}.
\end{align*}
In every case, we will say that an element $x$ of $C$ has \textit{rank} $i$ if $x\in\cO_i$.
\end{notation}

\begin{proposition} \label{orbits.pair}
Let $\cV_C$ be the Kantor pair associated with a Hurwitz algebra $C$, and consider the orbits of the action of $\Aut(\cV_C)$ on $\cV_C^+$ (or $\cV_C^-$). Then:\\
1) If $\dim C = 1$, then the orbits are $\cO_0$ and $\cO_1$. \\
2) If $\dim C > 1$, with either $\dim C \neq 2$ or $\chr\FF \neq 3$, then the orbits are $\cO_0$, $\cO_1$ and $\cO_2$. \\
Furthermore, in both cases, conjugate inverses are unique and only exist for elements in the orbit of nonzero norm.
\end{proposition}
\begin{proof}
We will only consider the subspace $\cV_C^+ = C$ (for $\cV_C^-$ the proof is analogous). It is clear that $\cO_0$ is always an orbit. 
It is easy to see that $\Aut(\cV_\FF) = \langle c_\lambda \med \lambda\in\FF^\times \rangle$, from where the classification follows in the case $\dim C = 1$.
Consider from now on the case $\dim C > 1$. 

We claim that there are at least $2$ nontrivial orbits. If $\chr\FF \neq 3$, for $e_1$ in a Cartan basis of $C$, it is easy to see that $\im U_{e_1} = \FF e_1$, $\im U_1 = C$, which have different dimensions, hence the claim follows. Now consider the case with $\chr\FF = 3$ and $\dim C > 2$. It is clear that $V_{1,1} = \id$ and $1$ has a conjugate inverse. If $e_1$ were invertible, we would have $V_{e_1, y} = \id$ for some $y\in C$; however, a straightforward computation using the Cartan basis shows that for the coordinates of the element $\{e_1, y, u_1\}$ in the Cartan basis, the coordinate corresponding to $u_1$ is $0$, so that $V_{e_1, y}(u_1) \neq u_1$ and $V_{e_1, y} \neq \id$; hence $e_1$ has no conjugate inverse, and therefore $1^+$ and $e_1^+$ are in different orbits. We have proven the claim.

Assume now that $\dim C = 2$. Then, for each $\lambda\in\FF^\times$, we have an automorphism $f_\lambda$ of $\cT_C$ as in \eqref{torus.f.lambda}.
Recall that $c_\lambda$ is an automorphism of $\cV_C$. Also, since $C$ is commutative, the involution defines a swapping automorphism $e_1 \leftrightarrow e_2$ on $C$. These automorphisms of $\cV_C$ show that any nonzero element of $C$ is either in the orbit of $e_1$ or in the orbit of $1$. The fact that there are at least $2$ nontrivial orbits proves the result in this case. We can assume from now on that $\dim C > 2$.

\medskip

Take $0\neq x\in C$ such that $n(x) = 0$. We claim that $x$ is in the orbit of the element $e_1$ for some Cartan basis of $C$.

Assume that $\lambda := n(x,1) \neq 0$. By applying the automorphism $c_{\lambda^{-1}}$, we can assume $n(x,1) = 1$. Then $e_1 := x$ and $e_2 := \overline{e_1} = 1 - e_1$ are isotropic orthogonal idempotents, and we can use the construction given in \cite{Eld98} (or in \cite[Chapter 4]{EKmon}) to complete them to a Cartan basis with multiplication table as in Figure~\ref{product.cartan.basis}. This proves the claim in this case.

Now suppose $n(x,1) = 0$. Then $x^2 = 0$ because $n(x) = 0 = t(x)$. Since the norm is nondegenerate, there exists $y\in C$ such that $n(x,y) = 1$. Thus $n(1, \overline{x}y) = 1$, and the elements $e_1:=\overline{x}y$ and $e_2:= \overline{e_1}=1-e_1$ are isotropic idempotents. Again, we can complete $\{e_1, e_2\}$ to a Cartan basis. Notice that $n(x,e_1) = n(x,\overline{x}y) = n(x^2,y) = 0$, because $x^2 = 0$, and $n(x,e_2)=n(x,1-e_1)=n(x,1)-n(x,e_1)=0$. Hence $x\in (\FF e_1 + \FF e_2)^{\perp}$. Furthermore, $xe_1 = x(\overline{x}y) = (x \overline{x})y=n(x)y = 0$ and $0 = \overline{x e_1} = \overline{e_1}$ $\overline{x} = e_2(-x)$, so $xe_1 = 0 = e_2x$. We have $xe_2 = x(1-e_1) = x$ and similarly $e_1 x = x$, so that $x e_2 = x = e_1 x$. We have that $x\in U: = e_1 C e_2$, and $(\FF e_1 + \FF e_2)^{\perp}= U\oplus V$ where $V:= e_2 C e_1$. One more time, using the construction in \cite{Eld98}, we can assume that $x = u_1$ in a Cartan basis. With the relation in \eqref{cartan.bases.relation}, equation \eqref{product.cd.by.cartan} shows that there is an automorphism of the form $L_z$, for some traceless element $z\in C$ of norm $1$, such that $L_z(x) \in \FF e_1$. Then, we can take $\lambda\in\FF^\times$ such that $c_\lambda L_z (x) = e_1$. We have proven the claim for the orbit of rank $1$ in all cases.

\medskip

Fix $x\in \cO_2$, that is, $n(x)\neq 0$. We now claim that $x$ and $1$ belong to the same orbit.

Take $\lambda\in\FF^\times$ such that $\lambda^2 n(x) = 1$; then $n(c^+_\lambda(x)) = 1$ and without loss of generality we can assume that $n(x) = 1$. If $n(x,1) = 0$, then $L_{\bar x}\in \Aut(\cV_C)$ and $L_{\bar x}(x) = \bar xx = n(x)1 = 1$, which proves the claim in this case. From now on, assume that $\lambda := n(x,1) \neq 0$. By the Cayley-Dickson doubling process applied to the algebra generated by $x$, there exists a traceless element $y\in (\FF 1 + \FF x)^\perp$ of norm $1$. Then, the element $z = L_y(x) = yx$ has norm $1$, and is also traceless because $n(z,1) = n(yx,1) = n(x, \bar y) = -n(x,y) = 0$. Consequently, $L_y$ and $L_{\bar z}$ are automorphisms and we have that $L_{\bar z} L_y (x) = 1$, which proves the claim.

\smallskip

Since there are at least $2$ nontrivial orbits in the case $\dim C > 2$, the classification follows for the remaining cases. The last statement follows from Remark~\ref{remark.conjugate.inverse} in the case $\chr\FF = 3$, and was proven in \cite{AF92} for the case $\chr\FF \neq 3$.
\end{proof}

\begin{proposition} \label{orbits.triple}
Let $\cT_C$ be the Kantor triple system associated with a Hurwitz algebra $C$, and consider the orbits of the action of $\Aut(\cT_C)$ on $C$. Then:\\
1) If $\dim C = 1$, the orbits are $\cO_0$ and $\cO_1(\lambda)$, with $\lambda\in\FF^\times$. \\
2) If $\dim C > 1$, the orbits are $\cO_0$, $\cO_1$ and $\cO_2(\lambda)$, with $\lambda\in\FF^\times$.
\end{proposition}
\begin{proof}
The result in the case $\dim C = 1$ follows easily from Corollary~\ref{co:aut_triple}.
The case with $\dim C = 2$ is easy to check using a Cartan basis $\{e_1, e_2\}$ and the automorphism group in Corollary~\ref{co:aut_triple} (which is generated by the automorphisms of the form $f_\lambda$ as in \eqref{torus.f.lambda}, and the swapping automorphism $e_1 \leftrightarrow e_2$). We can assume from now on that $\dim C > 2$. Note that $\cO_0$ is obviously an orbit in all cases.

\smallskip

Fix $0\neq x\in C$ such that $n(x) = 0$. We need to prove that $x$ is in the orbit of the element $e_1$ in a Cartan basis of $C$.

First, consider the case with $\lambda := n(x,1)\neq 0$. Then $e_1 := \lambda^{-1} x$ is an isotropic idempotent. As in Proposition~\ref{orbits.pair}, we can extend $e_1 $ to a Cartan basis, proving the claim in this case.

Now consider the case where $n(x, 1) = 0$. With the same arguments used in the proof of Proposition~\ref{orbits.pair}, we deduce that $x = u_1$ and that there is an automorphism $f \in\Aut(\cT_C)$ such that $f(x) \in \FF e_1$, where $e_1$ and $u_1$ belong to some Cartan basis of $C$. From the $2$-torus producing the Cartan grading on $C$, it follows that for each $\alpha\in\FF^\times$ there is some $g_\alpha\in\Aut(C)\leq\Aut(\cT_C)$ such that $g_\alpha (u_1) = \alpha u_1$. Therefore, $f(g_\alpha(x)) = \alpha f(x) \in \FF e_1$ for each $\alpha\in\FF^\times$, which forces $x$ and $e_1$ to be in the same orbit. We have proven the claim for all elements of rank $1$.

\smallskip

Consider the last case, that is, $x\in\cO_2$. Then $\mu := n(x) \neq 0$, and since $\FF$ is algebraically closed, we can take $\lambda\in\FF^\times$ such that $n(x) = \lambda^2$. Then we can write $x = \lambda y$ with $n(y) = 1$. With the same arguments used in the proof of Proposition~\ref{orbits.pair}, it follows that there is $f\in\Aut(\cT_C)$ such that $f(y) = 1$. Consequently, $f(x) = \lambda 1$, so that $x$ is in the orbit of $\lambda 1$. It follows that the elements of $\cO_2(\mu)$ are in the same orbit. 

Finally, by Theorem~\ref{th:schemestriple}, we have that $\Aut(\cT_C) \leq O(C, n)$, which forces the sets $\cO_1$ and $\cO_2(\lambda)$, for $\lambda\in\FF^\times$, to be contained in different orbits, and the result follows. 
\end{proof}

\smallskip

Now we will deal with the automorphisms and orbits of the most exceptional case in our list:

\begin{theorem} \label{th.aut.V_K.char3}
Let $K$ be a $2$-dimensional Hurwitz $\FF$-algebra, with $\chr\FF = 3$, and $\cV_K$ its associated Kantor pair.
Then there is a group isomorphism
$$\Psi\colon \Aut(\cV_K) \to \GL(K) \cong \GL_2(\FF), \quad (\varphi^+, \varphi^-) \mapsto \varphi^+,$$
where $\varphi^-$ is the dual inverse of $\varphi^+$ relative to the bilinear trace. 
Consequently, there is only one nontrivial $\Aut(\cV_K)$-orbit on $K^+$ (or $K^-$).
\end{theorem}
\begin{proof}
It is clear that $\Psi$ is a homomorphism. Let $\varphi\in\Aut(\cV_K)$. By \eqref{eq_Ux_Hurwitz} and since $\chr\FF = 3$, we have that
$$2n(x,y) \varphi^+(x) = \varphi^+\big(U^+_x(y)\big) = U^+_{\varphi^+(x)}\big(\varphi^-(y)\big) = 2n\big(\varphi^+(x), \varphi^-(y)\big) \varphi^+(x)$$
for all $x\in\cV^+_K$, $y\in\cV^-_K$, which implies that $n(x,y) = n\big(\varphi^+(x), \varphi^-(y)\big)$ for all $x\in\cV^+_K$, $y\in\cV^-_K$.
Hence $\varphi^-$ and $\varphi^+$ are dual inverses relative to the trace, and they determine each other. It follows that $\Psi$ is injective.

Since we assume that $\FF$ is algebraically closed, we have that there exists a Cartan basis $\{e_1, e_2\}$ of $K$, $K = \FF e_1 \oplus \FF e_2 \cong \FF\times\FF$, and the involution is given by $e_1 \leftrightarrow e_2$. It is straightforward to see that the triple product of $K$ on the Cartan basis is given by:
\begin{equation} \label{triple.product.K} \begin{aligned}
& \{e_1,e_2,e_1\} = 2e_1, \qquad \{e_2,e_1,e_1\} = -e_1, \\
& \{e_2,e_1,e_2\} = 2e_2, \qquad \{e_1,e_2,e_2\} = -e_2, \\
& \{e_i,e_i,e_i\} = \{e_1, e_1,e_2\} = \{e_2,e_2,e_1\} = 0.
\end{aligned}\end{equation}

Let $\tau\in\Aut(K,\inv)\leq\Aut(\cV_K)$ be the automorphism $\tau\colon e_1 \leftrightarrow e_2$ that swaps the idempotents of the Cartan basis (i.e., the involution). Fix $\alpha,\beta\in\FF^\times$ and consider the maps
\begin{equation*} \begin{aligned}
T^+_{\alpha,\beta} \colon \quad e_1^+ &\mapsto \alpha e_1^+, & \quad e_2^+ &\mapsto \beta e_2^+, \\
T^-_{\alpha,\beta} \colon \quad e_1^- &\mapsto \beta^{-1} e_1^-, & \quad e_2^- &\mapsto \alpha^{-1} e_2^-.
\end{aligned}\end{equation*}
By \eqref{triple.product.K}, it is clear that
$T_{\alpha,\beta} := (T^+_{\alpha,\beta}, T^-_{\alpha,\beta})$ is an automorphism of $\cV_K$.
Finally, fix $\lambda\in\FF$ and define $A_\lambda := (A^+_\lambda, A^-_\lambda)$ where 
\begin{equation} \label{autom_A_lambda}
A^\sigma_\lambda \colon \quad e_1 \mapsto e_1, \quad e_2 \mapsto e_2 + \sigma\lambda e_1,
\end{equation}
for $\sigma = \pm$. Then, using \eqref{triple.product.K}, it is straightforward to prove that $A_\lambda\in\Aut(\cV_K)$. For instance, we have that
\begin{align*}
\{ A^\sigma_\lambda (e_1), A^{-\sigma}_\lambda (e_2), A^\sigma_\lambda (e_1) \}
  &= \{ e_1, e_2 -\sigma \lambda e_1, e_1\} = \{e_1, e_2, e_1\} \\
	&= 2e_1 = A^\sigma_\lambda (2e_1) = A^\sigma_\lambda( \{e_1, e_2, e_1\} ), \\
\{ A^\sigma_\lambda (e_2), A^{-\sigma}_\lambda (e_2), A^\sigma_\lambda (e_2) \}
  &= \{ e_2 + \sigma \lambda e_1, e_2 - \sigma \lambda e_1, e_2 + \sigma \lambda e_1\} \\
	&= \{ \sigma \lambda e_1, e_2, e_2 + \sigma \lambda e_1\} +
	   \{ e_2, - \sigma \lambda e_1, e_2 + \sigma \lambda e_1\} \\
	&= -\sigma\lambda e_2 + 2\lambda^2 e_1 -2\sigma\lambda e_2 +\lambda^2 e_ 1 = 0 \\
	&=  A^\sigma_\lambda(0) = A^\sigma_\lambda (\{ e_2, e_2, e_2 \}),
\end{align*}
and the other cases are proven similarly. The maps of the form $\tau$, $T^+_{\alpha,\beta}$ and $A^+_\lambda$ act on the vector space $\cV^+_K$ as elementary matrices, and these generate a group isomorphic to $\GL_2(\FF)$; we can conclude that $\Psi$ is surjective, and therefore an isomorphism.
\end{proof}

\smallskip

\begin{theorem} \label{aut_Hurwitz_pair}
Let $C$ be a Hurwitz algebra and $\cV_C$ and $\cT_C$ the associated Kantor pair and triple system, respectively. Assume that either $\dim C \neq 2$ or $\chr\FF \neq 3$ (or both). Then, we have that
$$\Aut(\cV_C) = \langle c_\lambda, \Aut(\cT_C) \med\lambda\in\FF^\times \rangle \cong (\FF^\times \times \Aut(\cT_C))/C_2 $$
with $C_2 := \langle (-1, -\id) \rangle \cong \ZZ_2$. Consequently, each $\varphi\in\Aut(\cV_C)$ satisfies the property 
$$t(\varphi^+(x^+), \varphi^-(y^-)) = t(x^+,y^-)$$
for any $x^+\in\cV_C^+$, $y^-\in\cV_C^-$; equivalently, $(\varphi^+)^{-1} = \widehat{(\varphi^-)}$, where \, $\widehat{}$ \, denotes the dual relative to the bilinear trace.
\end{theorem}
\begin{proof}
First, consider the case with $\dim C > 2$. Fix $f = (f^+, f^-)\in\Aut(\cV_C)$. Recall that if $a\in C$ is traceless of norm 1, then $L_a \in \Aut(\cT_C) \subseteq \Aut(\cV_C)$. Notice that if $x\in\cV_C^+$ and $y\in\cV_C^-$ are conjugate inverses, then $g^+(x)$ and $g^-(y)$ are conjugate inverses for any automorphism $g$ of $\cV_C$. Since $1^+$ is invertible, $f^+(1^+)$ is invertible too, so that $\lambda \coloneqq n(f^+(1)) \neq 0$ (see Proposition~\ref{orbits.pair}). Without loss of generality, we can replace $f$ with $c_{1/\sqrt{\lambda}} f$ and assume that $n(f^+(1)) = 1$. Since $\dim C > 2$, the Cayley-Dickson doubling process shows that there is some traceless element $a\in C$ of norm $1$ such that $z \coloneqq L_a f^+(1)$ is traceless, and we also have $n(z) = 1$. Then, $L_{\bar z} L_a f^+(1) = 1$. Again, without loss of generality we can replace $f$ with $L_{\bar z} L_a f$ and assume that $f^+(1) = 1$. Since $f^-(1)$ must be a conjugate inverse of $1^+$ (which is unique by Remark~\ref{remark.conjugate.inverse}), we also have $f^-(1) = 1$. By Proposition~\ref{StabilizerIdentity}, it follows that $f := f^+ = f^- \in \Aut(C)$. Since the automorphisms $c_\lambda$ commute with any automorphism, it follows that $\Aut(\cV_C) = \langle c_\lambda, \Aut(\cT_C) \med\lambda\in\FF^\times \rangle$ and we have an epimorphism
\begin{equation} \label{isom_triple_to_pair} \begin{split}
\FF^\times \times \Aut(\cT_C) &\longrightarrow \Aut(\cV_C), \\
(\lambda, f) &\longmapsto c_\lambda f.
\end{split} \end{equation}
The kernel of the isomorphism in \eqref{isom_triple_to_pair} is $C_2$, which finishes the proof in this case.

Now, consider the case $\dim C = 2$. Fix $f = (f^+, f^-)\in\Aut(\cV_C)$. Again, composing $f$ with some automorphism of type $c_\lambda$ we can assume that $n(f^+(1)) = 1$. Then, there exists $\lambda\in\FF^\times$ such that $f^+(1) = \lambda^{-1} e_1 + \lambda e_2$, where $\{e_1, e_2\}$ is a Cartan basis of $C$. By \eqref{torus.f.lambda}, there is an automorphism $f_\lambda\in\Aut(\cT_C)$ such that $f_\lambda f(1^+) = 1^+$. Hence, replacing $f$ with $f_\lambda f$, we can assume that $f^+(1) = 1$. The same arguments used in the case above prove the result in this case (here we require $\chr\FF\neq 3$ in order to apply Proposition~\ref{StabilizerIdentity}).

Finally, for the case with $\dim C = 1$, the result is trivial because we have that
$\Aut(\cV_\FF) = \langle c_\lambda \med \lambda\in\FF^\times \rangle$ and $\Aut(\cT_\FF) = \langle c_\lambda \med \lambda = \pm 1\in\FF^\times \rangle$.

The last statement of the result follows from Theorem~\ref{th:schemestriple}, and because the automorphisms $c_\lambda$ behave well with the trace form.
\end{proof}

\smallskip

\begin{df}
Recall that the \textit{general orthogonal group} associated with an algebra $\cA$ with a norm form $n$ is the group of similarities of the norm, that is,
$$\mathrm{GO}(\cA, n):=\{ f\in \mathrm{GL}(\cA) \mid \exists \lambda \in \FF^{\times} \mbox{ such that }
n(f(x))=\lambda n(x) ~ \forall x\in \cA \} .$$
\end{df}

Next result classifies the automorphism groups for Kantor pairs of Hurwitz type.

\begin{corollary}\label{co:aut_pair}
Let $C$ be a Hurwitz algebra with either $\dim C \neq 2$ or $\chr\FF \neq 3$ (or both).\\
1) If $\dim C = 1$, then $\Aut(\cV_C) = \langle c_\lambda \med \lambda\in\FF^\times \rangle \cong \FF^\times$. \\
2) If $\dim C = 2$, then $\Aut(\cV_C) = GO(C, n) \cong (\FF^\times)^2 \rtimes \ZZ_2$. \\
3) If $\dim C = 4$, then $\Aut(\cV_C) \cong (\FF^\times \times O^+(C, n))/C_2$ with $C_2 \cong \ZZ_2$. \\
4) If $\dim C = 8$, then $\Aut(\cV_C) \cong \Gamma^+(\cS, n)$, where $\cS$ is the skew-symmetric subspace of $C$ and $\Gamma^+(\cS, n)$ denotes the associated even Clifford group.
\end{corollary}
\begin{proof}
Consequence of Theorem~\ref{aut_Hurwitz_pair} and Corollary~\ref{co:aut_triple}.
\end{proof}

\smallskip

\begin{df}
Let $\cA$ be a graded algebra. Recall that a bilinear form $b:\cA\times \cA\rightarrow \FF$ is said to be \textit{homogeneous} if we have $g + h = 0$ whenever $b(\cA_g,\cA_h)\neq 0$. (Note that this definition is consistent with the fact that the only grading up to equivalence on $\FF$ is the trivial grading: $\FF_0 := \FF$.) The definition is analogous for a bilinear form on a graded Kantor triple system. On the other hand, given a graded Kantor pair $\cV$, a bilinear form $b:\cV^+ \times \cV^- \rightarrow \FF$ will be called \textit{homogeneous} if we have $g + h = 0$ whenever $b(\cV^+_g,\cV^-_h)\neq 0$.
\end{df}

The following Lemma can be regarded as an extension from the case of Hurwitz algebras (see \cite[proof of Proposition 4.10]{EKmon}) to the case of Kantor pairs and triple systems of Hurwitz type.

\begin{lemma}\label{le:homogeneous.trace}
Let $C$ be a Hurwitz algebra and consider its associated Kantor pair $\cV_C$ and triple system $\cT_C$.
Assume also that either $\dim C \neq 2$ or $\chr\FF \neq 3$.
Then, for any grading on $\cV_C$ (resp. $\cT_C$), the bilinear (resp. linear) trace of $C$ is homogeneous on $\cV_C$ (resp. $\cT_C$).
\end{lemma}
\begin{proof}
Recall that any grading $\Gamma$ on $\cT_C$ induces a grading $(\Gamma, \Gamma)$ on $\cV_C$. Also, the linear trace can be recovered from the bilinear trace via $t(x) = t(x, 1)$. Therefore, we only need to prove the result in the case of $\cV\coloneqq\cV_C$. Fix a $G$-grading $\Gamma$ on $\cV$ and let $x\in \cV^+_g$, $y\in \cV^-_h$ be such that $t(x, y) \neq 0$. We need to prove that $\deg^+(x) + \deg^-(y) = 0$. Without loss of generality, we can replace $y$ by $2 t(x, y)^{-1} y$ and assume that $t(x, y) = 2$. We will denote $\cH \coloneqq \cH(C, \inv)$ and $\cS \coloneqq \cS(C, \inv)$.

First consider the case that $x$ is in the orbit of rank $2$. Note that for any automorphism $f$, by Theorem~\ref{aut_Hurwitz_pair} we have that $t(f^+(x), f^-(y)) = t(x, y)$; hence, up to automorphism (by Proposition~\ref{orbits.pair}), we can assume that $x = 1^+$. It follows that $y = 1 + s$ with $s\in\cS$. In case $\chr\FF = 3$, we have that $U_1^+(s) = 0$, so that $U_1^+(y) = 1$ and taking degrees we get $\deg^+(x) + \deg^-(y) = 0$. Now consider the case $\chr\FF\neq 3$; thus $U_1$ is invertible. Since $U^+_1$ is invertible and homogeneous with $U^+_1(1^-) = 1^+$, it follows that $1^-$ is homogeneous and $\deg(1^+) + \deg(1^-) = 0$. On the other hand, $z^- := U^-_1 U^+_1(y) = U^-_{1^-}(1 - 3s) = 1 + 9s \in \cV^-_h$. Since $y = 1 + s$, $1$ and $1 + 9s$ are linearly dependent and homogeneous in $\cV^-$, they must have the same degree. Thus $\deg^-(y) = \deg^-(1) = -\deg^+(x)$.

Finally, consider the case where $x$ has rank $1$. This time, we can assume without loss of generality that $x = e_1^+$ where $e_1$ denotes the corresponding idempotent in a Cartan basis of $C$. From $t(x, y) = 1$ we deduce that $y = e_2 + z$ with $z \in (\FF e_1)^\perp = \lspan\{ e_1, u_i, v_i \med i = 1,2,3 \}$. But we also have that $\ker U^+_x = \lspan\{ e_1^-, u_i^-, v_i^- \med i =1,2,3 \}$. We conclude that $z \in \ker U_{e_1}^+$, so that $U_x^+(y) = U^+_{e_1}(e_2 + z) = 2 e_1^+ = 2 x$ and again $\deg^+(x) + \deg^-(y) = 0$.
\end{proof}

\section{Classification of fine gradings on Hurwitz Kantor systems} \label{section.classification.gradings}

In this section, we will first describe some important Peirce decompositions on Hurwitz pairs. Secondly, some examples of fine gradings by their universal groups on Kantor pairs and triple systems of Hurwitz type will be given. And finally, we will prove that any fine grading on a Kantor pair or triple system of Hurwitz type is equivalent to exactly one of those examples.

\subsection{Peirce decompositions on Hurwitz pairs}

\begin{df}
Given a Hurwitz algebra $C$ with $\dim C = 2^m > 1$, with two copies of the basis in \eqref{second.cartan.basis} we get a basis of $\cV_C$ that will be denoted $B_{\ZZ}(\cV_C)$ and called a \emph{Cartan basis} of $\cV_C$. Note that $B_{\ZZ}(\cV_C)$ consists of the idempotents of $\cV_C$ given by $e_g^\alpha \coloneqq (v_g^\alpha, v_g^{\omega\alpha})$ for $g\in\ZZ_2^{m-1}$, $\alpha\in\widehat{\ZZ}_2 = \langle \omega \rangle$.
\end{df}

\begin{proposition} \label{peirce.cartan}
Let $C$ be a Hurwitz algebra with $\dim C = 2^m > 1$. Then, the idempotents $e_g^{\alpha}$ of the Cartan basis of $\cV_C$ produce a simultaneous Peirce decomposition that is given by:
\begin{equation*} \begin{aligned}
\cV^\sigma_{1,1}(e_g^\alpha) &= \FF (e_g^\alpha)^\sigma, \; &
\cV^\sigma_{-\frac{1}{2},0}(e_g^\alpha) &= \FF (e_g^{\omega\alpha})^\sigma, \\
\cV^\sigma_{0,1}(e_g^\alpha) &= \bigoplus_{g\neq h\in\ZZ_2^{m-1}} \FF (e_h^\alpha)^\sigma, \; &
\cV^\sigma_{\frac{1}{2},\frac{1}{2}}(e_g^\alpha) &= \bigoplus_{g\neq h\in\ZZ_2^{m-1}} \FF (e_h^{\omega\alpha})^\sigma,
\end{aligned} \end{equation*}
for $\sigma = \pm$.
\end{proposition}
\begin{proof}
Consider first the case with $m=3$, i.e., $C = \cC$.
The character $\omega\in\widehat{\ZZ}_2$ (as in Notation~\ref{NotationCartan}), identified with its natural extension $\omega \times \mathds{1} \in \widehat{\ZZ}_2 \times \widehat{\ZZ_2^2} \cong \widehat{\ZZ_2^3}$, produces an automorphism $f_\omega$ of the $\ZZ_2^3$-grading on $\cC$ that is given on the Cayley-Dickson basis via $f_\omega(x_g) = \omega(g)x_g$. It is easy to see that $f_\omega$ permutes the subspaces $\FF e_g^\alpha \leftrightarrow \FF e_g^{\alpha\omega}$ for each $g\in\ZZ_2^2$ and $\alpha\in\widehat{\ZZ}_2$. (Note that the automorphism $f_\omega$ of $\cC$ appears in \cite[Proof of Th.~4.17]{EKmon} with a different guise.) Also, recall from Lemma~\ref{lemma.automorphisms} that for each element $x_g$ of the Cayley-Dickson grading, $L_{x_g}$ defines an automorphism, which is given by \eqref{product.cd.by.cartan}. By the symmetries given by these automorphisms, it is clear that it suffices to prove the result for the case $e = e_0^\mathds{1}$, which is easy to calculate using \eqref{cartan.bases.relation} and the products of the Cartan basis $B_\ZZ(\cC)$ of $\cC$. For the cases $m=1,2$, the result follows by restriction to the corresponding subspaces.
\end{proof}

\begin{remark}
The left Peirce decompositions in Proposition~\ref{peirce.cartan} are produced by $D$-operators, which are derivations, so that we have gradings associated with them. Namely, we have a grading by $G = \frac{1}{2}\ZZ$ determined by $\cV^\sigma_{\sigma\lambda} := \cV^\sigma_{\lambda, \mu}(e^\alpha_g)$, and therefore, by composition with the group isomorphism $\frac{1}{2}\ZZ \to \ZZ$, $\frac{1}{2} \mapsto 1$, we get a $\ZZ$-grading with homogeneous subspaces given by
\begin{equation} \label{peirceZgrad}
\cV^\sigma_{\sigma 2\lambda} := \cV^\sigma_{\lambda, \mu}(e^\alpha_g).
\end{equation}
Also, note that the grading in \eqref{peirceZgrad} can be properly refined with a shift on the degree, or just combining it with the $\ZZ$-grading given by $\cV^\sigma_{\sigma 1} = \cV^\sigma$.
\end{remark}

\subsection{Examples of fine gradings on Hurwitz Kantor systems}

\begin{example}\label{examples.dim.one}
For the $1$-dimensional Hurwitz algebra $C=\FF$, it is clear that up to equivalence we have only one fine grading on $\cV_C$; its universal group is isomorphic to $\ZZ$, and its universal degree is equivalent to $\deg (1^+) = -\deg(1^-) = 1$.
Similarly, it is easy to see that up to equivalence there is only one fine grading on $\cT_C$, whose universal group is $\ZZ_2$, and its associated universal degree is equivalent to $\deg (1) = \bar{1}$.
\end{example}

\begin{example} \label{example.pair.cd}
Let $C$ be a Hurwitz algebra of dimension $2^m$ for some $m\in \{1,2,3\}$, where we also require to exclude the case with $\chr\FF = 3$ and $\dim C = 2$ (the grading obtained in that case with the construction below is not given by its universal group, and is actually equivalent to a grading given in Example~\ref{example.pair.cartan}).
Set $G = \ZZ_2^m$ and take a Cayley-Dickson basis $B_\CD(C) = \{ x_g \}_{g\in G}$ of $C$.
There is a grading $\Gamma_\CD(\cV_C)$ by the group $\ZZ \times G = \ZZ \times \ZZ_2^m$ on the Kantor pair $\cV_C$, where $B_\CD(C)$ is a homogeneous basis in $\cV_C^\sigma$ for $\sigma = \pm$, and that is given by
\begin{equation}
\deg(x_g^\sigma) \coloneqq (\sigma 1, g)
\end{equation}
for $\sigma = \pm$. The grading $\Gamma_\CD(\cV_C)$ will be called the \emph{Cayley-Dickson grading} on $\cV_C$.
\end{example}

\begin{example} \label{example.triple.cd}
Let $C$, $G$ and $B_\CD(C)$ be as in Example~\ref{example.pair.cd}.
Then, we have a grading $\Gamma_\CD(\cT_C)$ by the group $\ZZ_2 \times G = \ZZ_2^{m+1}$ on the Kantor triple system $\cT_C$, with homogeneous basis $B_\CD(C)$, and determined by
\begin{equation}
\deg(x_g) \coloneqq (\bar 1, g).
\end{equation}
We will refer to $\Gamma_\CD(\cT_C)$ as the \emph{Cayley-Dickson grading} on $\cT_C$.
\end{example}

\begin{example} \label{example.pair.cartan}
Now we will define a grading on the Cayley pair $\cV_\cC$ by the group $\ZZ^4$ using two copies of a Cartan basis of $\cC$. Define the following map:
\begin{equation} \begin{aligned}
\deg(e_1^+) &= -\deg(e_2^-) := (1,0,0,0), \\
\deg(e_2^+) &= -\deg(e_1^-) := (0,1,0,0), \\
\deg(u_1^+) &= -\deg(v_1^-) := (0,0,1,0), \\
\deg(u_2^+) &= -\deg(v_2^-) := (0,0,0,1), \\
\deg(u_3^+) &= -\deg(v_3^-) := (1,2,-1,-1), \\
\deg(v_1^+) &= -\deg(u_1^-) := (1,1,-1,0), \\
\deg(v_2^+) &= -\deg(u_2^-) := (1,1,0,-1), \\
\deg(v_3^+) &= -\deg(u_3^-) := (0,-1,1,1).
\end{aligned} \end{equation}
Note that the second coordinate of $\deg$ coincides with the $\ZZ$-grading in \eqref{peirceZgrad} associated with the idempotent $e^{\charTrivial}_{a_2+a_3} = \frac{-1}{\sqrt{2}}(u_3, v_3)$ of $\cV_\cC$. On the other hand, the last two coordinates of $\deg$ coincide with the Cartan grading on $\cC$ given in \eqref{cartan.degree.C}, which extends to a grading on $\cV_\cC$. Furthermore, the sum of all coordinates of $\deg$ is the $\ZZ$-grading given by $\cV^\sigma_{\sigma 1} :=\cV^\sigma$, and since a linear combination of compatible gradings defines a grading, it follows that the first coordinate of $\deg$ is also a grading on $\cV_\cC$. Consequently, $\deg$ defines a grading, which will be referred to as the \emph{Cartan grading} on $\cV_\cC$. Analogous Cartan gradings are defined for Hurwitz pairs of dimensions $2$ and $4$ by restriction of the same degree map to the bases $\{e_1, e_2\}$ and $\{e_1, e_2, u_1, v_1\}$, where the grading groups are $\ZZ^2$ and $\ZZ^3$, respectively. It is obvious that the associated Cartan basis of $\cV_C$ is homogeneous. Given a Hurwitz algebra $C$ with $\dim C > 1$, the Cartan grading on $\cV_C$ will be denoted by $\Gamma_{\ZZ}(\cV_C)$.
\end{example}

\begin{example} \label{example.triple.cartan}
Consider the Cartan grading on $\cV_\cC$. If we force the relation $\deg(x^+) = \deg(x^-)$ for $x$ in the Cartan basis of $\cC$, we obtain a grading by $\ZZ^3$ on $\cV_\cC$ which also restricts to a grading on $\cT_\cC$. This grading will be called the \emph{Cartan grading} on $\cT_\cC$ and denoted by $\Gamma_\ZZ(\cT_\cC)$, and its degree map is given by:
\begin{equation} \begin{aligned}
\deg(e_1) &= -\deg(e_2) := (1,0,0), \\
\deg(u_1) &= -\deg(v_1) := (0,1,0), \\
\deg(u_2) &= -\deg(v_2) := (0,0,1), \\
\deg(u_3) &= -\deg(v_3) := (-1,-1,-1).
\end{aligned} \end{equation}
For a Hurwitz algebra $C$ of dimension 2 or 4, we similarly define its Cartan grading $\Gamma_\ZZ(\cT_C)$ by restriction of $\deg$ to the corresponding basis $\{e_1, e_2\}$ or $\{e_1, e_2, u_1, v_2\}$; in these two cases, the grading groups are $\ZZ$ and $\ZZ^2$, respectively.
\end{example}

\begin{proposition}
The gradings given in Examples \ref{examples.dim.one}, \ref{example.pair.cd}, \ref{example.triple.cd}, \ref{example.pair.cartan} and \ref{example.triple.cartan} are fine and given by their universal groups and their universal degrees.
\end{proposition}
\begin{proof}
It is clear that all these gradings are fine because their homogeneous components are $1$-dimensional. For the $1$-dimensional cases, the result is trivial. For the Cayley-Dickson gradings, the result follows from Propositions \ref{regulargradspairs} and \ref{regulargradstriples} (note that the Kantor pair case with $\dim C = 2$ and $\chr\FF = 3$ is excluded, hence the requirements to apply Proposition~\ref{regulargradspairs} hold).

Now consider the Cartan grading on the Cayley pair and let $G$ and $\deg$ denote its universal group and degree, respectively. By Lemma~\ref{le:homogeneous.trace}, the bilinear trace is homogeneous, from where it follows that $\deg$ is determined by its restriction to $\cV^+_\cC$ (or $\cV^-_\cC$), and also that the following relations hold: $\deg(e_1^\sigma) + \deg(e_2^{-\sigma}) = 0$ and $\deg(u_i^\sigma) + \deg(v_i^{-\sigma}) = 0$ for $i\in\{1,2,3\}$ and $\sigma=\pm$. Denote $a := \deg(e_1^+)$, $b := \deg(e_2^+)$, $c := \deg(u_1^+)$, $d := \deg(u_2^+)$. From the relations $0\neq \{e_1^+, u_1^-, e_2^+\}\in\FF u_1^+$, $0\neq \{e_1^+, u_2^-, e_2^+\}\in\FF u_2^+$, $0\neq \{e_1^+, u_1^-, u_2^+\}\in\FF v_3^+$ and $0\neq \{e_1, u_2, e_2\}\in\FF u_2$ we get that $\deg(v_1^+) = a+b-c$, $\deg(v_2^+) = a+b-d$, $\deg(v_3^+) = -b+c+d$ and $\deg(u_3^+) = a+2b-c-d$. Since $G$ is generated by its support, it is also generated by $\{a,b,c,d\}$. By the universal property, the map $f\colon G \to \ZZ^4$ sending $\{a,b,c,d\}$ to the canonical basis of $\ZZ^4$ is a group epimorphism inducing the Cartan grading; this also implies that $f$ is actually a group isomorphism and that the Cartan grading is given by its universal degree. (Note that this proof avoids computing the triple product for all the $8^3$ possible cases to check all relations between generators.) For the Cartan gradings on Hurwitz pairs of dimensions $2$ and $4$, the result follows from the same arguments.

Finally, consider the Cartan grading on the Cayley triple system $\cT_\cC$; let $H$ and $\deg$ denote its universal group and degree. Note that $H$ is defined exactly by the same relations as the universal group $G$ of the Cartan grading on $\cV_\cC$ and the additional relation given by $\deg(x^+) = \deg(x^-)$. It is easy to see that the last relation is equivalent to $a + b = 0$, from where it follows that $H \cong \ZZ^3$ and that the Cartan grading on $\cT_\cC$ is given by its universal degree. Again, the same arguments hold for Hurwitz triple systems of dimensions $2$ and $4$.
\end{proof}

\subsection{Classification of fine gradings on Hurwitz Kantor pairs}

\begin{remark}
Recall that in the case where $\dim C=1$ we have only one fine grading on $\cV_C$ (see Example~\ref{examples.dim.one}). So it remains to deal with the classification in the cases with $\dim C > 1$.
\end{remark}

\begin{theorem}\label{th:finegradspairs}
Let $\Gamma$ be a fine grading on $\cV_C$ and $\dim C\geq 2$. Then:
\begin{itemize}
\item[$1)$] If $\chr\FF = 3$ and $\dim C = 2$, then $\Gamma$ is, up to equivalence, the Cartan grading.
\item[$2)$] If either $\chr\FF \neq 3$ or $\dim C \neq 2$, then $\Gamma$ is, up to equivalence, the Cayley-Dickson grading or the Cartan grading.
\end{itemize}
\end{theorem} 
\begin{proof}
Let $\Gamma$ be a fine $G$-grading on $\cV_C$. We will deal first with the case 1), so assume that $\chr\FF = 3$ and $\dim C = 2$.
By Theorem~\ref{th.aut.V_K.char3}, the group $\Aut(\cV_C) \cong \GL_2(\FF)$ acts transitively on the set of bases of $\cV_C^+$, so that by applying an automorphism we can assume that a Cartan basis $\{e_1, e_2\}$ of $C$ is homogeneous in $\cV_C^+$. Then, the subspaces $\ker U_{e_i}^+ = \FF e_i^-$, for $i = 1,2$, are graded, so that the Cartan basis is also homogeneous in $\cV_C^-$. Consequently, since $\Gamma$ is fine, it must be equivalent to the Cartan grading.

From now on, consider the case with either $\chr\FF \neq 3$ or $\dim C \neq 2$.
Assume that $\dim C = 8$; since the cases with $\dim C = 2,4$ are anologous, we will omit the details.
Consider the case where there exists a homogeneous element $x^+\in \cV_C^+$ of rank $2$ (for the case $x^-\in \cV_C^-$, the proof is analogous). By applying an automorphism we can assume that $x^+ = 1^+$. Take $g = -\deg(1^+)$. By Proposition~\ref{regulargradspairs}, the shifted grading $\Gamma^{[g]}$ restricts to a grading $\Gamma_C$ on $C = \cV_C^+$. Notice that $(\Gamma^{[g]})^+$ and $(\Gamma^{[g]})^-$ have the same homogeneous components, although the degrees are not necessarily the same. Since $\Gamma$ is fine, $\Gamma_C$ is fine on $C$. Then $\Gamma_C$ cannot be equivalent to the Cartan grading on $C$ because its extension admits a proper refinement on $\cV_C$ (namely, the Cartan grading on $\cV_C$). Thus $\Gamma_C$ must be equivalent to the Cayley-Dickson grading on $C$, and consequently, $\Gamma$ is equivalent to the Cayley-Dickson grading on $\cV_C$.

\smallskip

Now consider the case where there are no homogeneous elements of rank $2$, that is, all homogeneous elements have rank $1$ (which forces them to be isotropic). Since $n$ is nondegenerate, we can take two homogeneous elements $x,y\in \cV^+_C$ such that $n(x, y) \neq 0$. Since $n(x,y) = n(x + y) - n(x) - n(y) = n(x + y)$, we also have that $n(x + y) \neq 0$. Up to automorphism (see Proposition~\ref{orbits.pair}), we can assume that $x + y = 1^+$. Then $2 = t(1^+)= t(x) + t(y)$, and so we have either $t(x)\neq 0$ or $t(y)\neq 0$. Without loss of generality, consider the case $t(x)\neq 0$. Since $x$ is isotropic, we have $x^2 = \lambda x$ for $\lambda = t(x)\neq 0$ and $e_1:= \lambda^{-1} x$ is an isotropic idempotent. Using the same arguments as in \cite[Chapter 4]{EKmon}, we can complete $e_1$ to a Cartan basis where $e_2: = 1 - e_1$ is an isotropic idempotent. Since $y = 1 - x = (e_1 + e_2) - \lambda e_1 = (1 - \lambda)e_1 + e_2$ is isotropic it follows that $\lambda = 1$. Therefore $x = e_1$ and $y = e_2$ are homogeneous in $\cV_C^+$.

Notice that $U_{e_1^+, e_2^+}$ and $U_{e_i^+}$, for $i=1,2$, are homogeneous maps; hence we have that $\mathrm{Ker} (U_{e_1^+, e_2^+})= \FF e_1 \oplus \FF e_2$ and $\mathrm{Ker} (U_{e_i^+})=\FF e_i \oplus U \oplus V$ are graded subspaces in $\cV^-$ for $i = 1,2$, where $U := \text{span}\{ u_i \med i = 1,2,3 \}$ and $V := \text{span}\{ v_i \med i = 1,2,3 \}$. Intersecting these subspaces we get that $U\oplus V$, $\FF e_1$ and $\FF e_2$ are graded subspaces in $\cV^-$. Analogously, $U\oplus V$ is graded in $\cV^+$. Now we will prove that $U$ and $V$ are graded subspaces in $\cV^{\sigma}$ for $\sigma = \pm$. Let $\deg(e_i^+) = g_i$ for $i = 1,2$ and notice that, since the trace is homogeneous (Lemma~\ref{le:homogeneous.trace}), we have $\deg(e_1^-) = -g_2$ and $\deg(e_2^-) = -g_1$. Take $0\neq u + v \in (U\oplus V)^-_g$ for some $g \in G$, with $u\in U$, $v\in V$. Then $\{ e_1^+, u+v,e_2^+ \} = -(u+2v) \in (U \oplus V)^+_{g+g_1+g_2}$ and $\{ e_1^-, u+2v,e_2^-\} = -(u+4v) \in (U \oplus V)^-_{g}$. It follows that $v$ and $u$ are homogeneous of degree $g$ in $\cV_C^-$. Therefore $U$ and $V$ are graded in $\cV^+$, and similarly in $\cV^-$. If $0\neq u\in U^-_g$ for some $g\in G$, then $\{ e_1^+, u,e_2^+ \} = -u \in U^+_{g_1+g_2+g}$, and it follows that the homogeneous components of $U^+$ and $U^-$ coincide. Similarly, the homogeneous components of $V^+$ and $V^-$ coincide. 

Following again \cite[Chapter 4]{EKmon}, we will now construct a homogeneous Cartan basis. Denote $U := e_1 C e_2$ and $V := e_2 C e_1$. Take a homogeneous basis $\{ u_i \}_{i=1}^3$ of $U^+$ such that $n(u_1u_2, u_3) = 1$ (this is possible because $n(U^2,U)\neq 0$ and $n(U) = 0$). Note that $\{ u_i \}_{i=1}^3$ are homogeneous in $U^-$ too. Then $\{ v_1 = u_2u_3, v_2 = u_3u_1, v_3 = u_1u_2 \}$ is the dual basis in $V$. Since $\{ e_1^{\sigma}, u_i^{- \sigma},u_j^{\sigma} \}$ are homogeneous for $i,j = 1,2,3$, it is clear that $v_i$ is homogeneous in $\cV_C^\sigma$ for $i = 1,2,3$ and $\sigma = \pm$. We have obtained a homogeneous Cartan basis of $\cV_C$. Since $\Gamma$ is fine, it must be equivalent to the Cartan grading on $\cV_C$. 
\end{proof}

\subsection{Classification of fine gradings on Hurwitz Kantor triple systems}

\begin{remark}
Recall from Corollary~\ref{co:equiv.grads.classif} that the classification of fine gradings coincide on $\cT_C$ and $\cT'_C$. It turns out that fine gradings on the triple system $\cT'_\cC$, where $\cC$ is a Cayley algebra, have also been classified in an independent work \cite{DET21}, where this triple system has been used to classify gradings on a $3$-fold cross product denoted by $(\cC, X^\cC_1)$.
\end{remark}

\begin{remark}
We already know that in the case where $\dim C=1$ there is only one fine grading on $\cT_C$ (see Example~\ref{examples.dim.one}). We will now deal with the remaining cases.
\end{remark}

\begin{theorem}
Let $\Gamma$ be a fine grading on $\cT_C$ and $\dim C\geq 2$. Then $\Gamma$ is, up to equivalence, the Cayley-Dickson grading or the Cartan grading.
\end{theorem} 
\begin{proof}
Let $\Gamma$ be a fine $G$-grading on $\cV_C$. Only the case $\dim C = 8$ will be considered (the cases with $\dim C = 2,4$ are proven with the same arguments).

Consider first the case where there is some homogeneous element $x$ in some orbit $\cO_2(\lambda)$. Since $\FF$ is algebraically closed, we can scale $x$ and assume that $n(x) = 1$. Furthermore, up to automorphism of $\cT_C$, we can assume that $x = 1$ is homogeneous. By Proposition \ref{regulargradstriples}, for $g=\deg(1)$, $\Gamma^{[g]}$ restricts to a grading $\Gamma_C$ on $C$, which is fine because $\Gamma$ is fine. With exactly the same argument used in the first part of the proof of Theorem~\ref{th:finegradspairs}, it follows that $\Gamma$ is equivalent to the Cayley-Dickson grading on $\cT_C$. 

\smallskip

Now consider the case where all homogeneous elements are in the orbit of rank $1$. Since $n$ is nondegenerate we can take two homogeneous elements $x,y\in \cT_C$ such that $n(x + y)\neq 0$. Following the proof of Theorem \ref{th:finegradspairs}, we deduce that, up to automorphism, $x$ and $y$ are isotropic idempotents in $C$, which will be denoted by $e_1$ and $e_2$. By Theorem \ref{th:schemestriple}, we can use the product of $\cT_C'$ instead of the one in $\cT_C$, because the classification of fine gradings (and the orbits under their automorphisms groups) coincide on both triple systems. Denote $U := e_1 C e_2$ and $V := e_2 C e_1$. Since $e_1$ and $e_2$ are homogeneous, $\{ e_1, C, e_2 \}' = U$ and $\{ e_2, C, e_1 \}' = V$ are graded subspaces. Following the proof of Theorem \ref{th:finegradspairs}, we can take a homogeneous basis $\{ u_i \}_{i=1}^3$ of $U$ such that $n(u_1 u_2, u_3) = 1$, and construct its dual basis $\{ v_i \}_{i=1}^3$ in $V$, so that we obtain a homogeneous Cartan basis of $\cT_C$. Since $\Gamma$ is fine, it must be equivalent to the Cartan grading on $\cT_C$. 
\end{proof}

\section{Weyl groups of fine gradings on Hurwitz Kantor systems} \label{section.weyl}
In this section we compute the Weyl grups of the fine gradings on Kantor pairs and triple systems of Hurwitz type.

\begin{remark}
Consider the Kantor pair $\cV_\FF$ and triple system $\cT_\FF$ associated with the $1$-dimensional Hurwitz algebra $\FF$. Let $\Gamma$ be the only fine grading on $\cV_\FF$ or on $\cT_\FF$. In both cases, it is clear that $\cW(\Gamma)$ is the trivial group. Thus, it remains to deal with the Kantor pairs and triple systems associated with Hurwitz algebras of dimension greater than $1$.
\end{remark}

\begin{theorem}\label{weyl.CD} Let $C$ be a Hurwitz algebra of dimension $2^m$ for some $m\in\{1,2,3\}$, and $\Gamma_{\CD}$ the Cayley-Dickson $\ZZ\times\ZZ_2^m$-grading (resp., $\ZZ_2^{m+1}$-grading) on $\cV_C$ (resp., on $\cT_C$). Then we have:
$$
\cW(\Gamma_{\CD}) \cong \left\{\left( \begin{array}{c|c}
1 & 0 \\ \hline a & A
\end{array}\right) | \; a\in\ZZ_2^m, \; A\in\GL_m(\ZZ_2) \right\} \lesssim \GL_{m+1}(\ZZ_2).
$$
\end{theorem}
\begin{proof}
Denote the Cayley-Dickson grading on $\cV_C$ (resp. on $\cT_C$) by $\Gamma_\cV$ (resp. by $\Gamma_\cT$). Also denote
$$\cW \coloneqq
\left\{\left( \begin{array}{c|c}
1 & 0 \\ \hline a & A
\end{array}\right) \med \; a\in\ZZ_2^m, \; A\in\GL_m(\ZZ_2) \right\}.$$
For the case $m = 3$, recall from \cite{EKmon} that if $\Gamma_C$ is the fine $\ZZ_2^m$-grading on $C$, then
$\cW(\Gamma_C) \cong \Aut(\ZZ_2^m) \cong \GL_m(\ZZ_2)$. By restriction of the automorphisms, the same holds for the cases $m = 1,2$.

First, consider the case on the Kantor pair $\cV_C$ with $m = 3$, so that $C = \cC$ is the Cayley algebra. Since the automorphisms of $\cC$ extend to $\cV_\cC$, we can identify $\cW(\Gamma_\cC) \leq \cW(\Gamma_\cV)$. Therefore, $\cW(\Gamma_\cV)$ has a subgroup corresponding to the block structure
$$ 
G_\cC \coloneqq \left\{\left( \begin{array}{c|c}
1 & 0 \\ \hline 0 & A \end{array}\right) \med A\in \GL_3(\ZZ_2) \right\} \cong \GL_3(\ZZ_2).
$$
Now, fix a traceless element $x\in\cC$ of norm $1$ that is homogeneous in $\cV_\cC^+$. Then, $x$ is also homogeneous in $\cV_\cC^-$ and, by Lemma~\ref{lemma.automorphisms}, the map $L_x$ is a homogeneous automorphism of $\Gamma_\cV$. Moreover, $L_x$ induces an element of $\cW(\Gamma_\cV)$ corresponding to a block of the form
$$
M = \left( \begin{array}{c|c}
1 & 0 \\ \hline a & A \end{array}\right)
$$
for some $0 \neq a \in\ZZ_2^3$, $A\in\GL_3(\ZZ_2)$. It is easy to see that the group generated by $M$ and $G_\cC$ is $\cW$, which implies the inclusion $\cW \lesssim \cW(\Gamma_\cV)$. Furthermore, we have
$$
\cW(\Gamma_\cV) \leq \{\varphi\in\Aut(\ZZ\times\ZZ_2^3) \med \varphi(\supp\Gamma_\cV^\sigma)
= \supp\Gamma_\cV^\sigma, \; \sigma = \pm \} \equiv \cW.
$$
We have proven the isomorphism $\cW(\Gamma_\cV) \cong \cW$.

The same arguments above prove the result for the Kantor pairs in the cases $m = 1, 2$, and also that $\cW \lesssim \cW(\Gamma_\cT)$ in all the cases for triple systems. Since we have the natural inclusion $\cW(\Gamma_\cT) \leq \cW(\Gamma_\cV) \cong \cW$, the result follows for triple systems too.
\end{proof}

\begin{theorem} \label{weyl.Cartan}
Let $C$ be a Hurwitz algebra with $\dim C = 2^m > 1$. Then,
$$\cW(\Gamma_\ZZ(\cV_C)) \cong \cW(\Gamma_\ZZ(\cT_C)) \cong \Sym(\ZZ_2^{m-1}) \times \Sym(\widehat{\ZZ}_2) \cong \Sym(2^{m-1}) \times\ZZ_2 \lesssim \Aut\Phi,$$
where $\Phi$ denotes the root system of $\kan(C)$.

Here, an element $(\rho, \tau) \in \Sym(\ZZ_2^{m-1}) \times \Sym(\widehat{\ZZ}_2)$ corresponds to the permutation of subspaces $\FF v_g^\alpha \mapsto \FF v_{\rho(g)}^{\tau(\alpha)}$ in $\cT_C$ or $\cV_C^+$.
\end{theorem}
\begin{proof}
We already know that Weyl groups extend as subgroups when extending gradings from Kantor triple systems to Kantor pairs, and from Kantor pairs to their Kantor-Lie algebras, so that we have $\cW(\Gamma_\ZZ(\cT_C)) \leq \cW(\Gamma_\ZZ(\cV_C)) \leq \Aut\Phi$.

From the automorphisms given in \cite[Proof of Th.~4.17]{EKmon}, it follows that for each
$\rho\in\Sym(\ZZ_2^{m-1}\setminus\{\bar 0\})$ and $\tau\in\Sym(\widehat{\ZZ}_2)$ there exists $\varphi\in\cW(\Gamma_\ZZ(\cT_C))$ acting via
$\varphi(\FF v^\alpha_g) = \FF v^{\tau(\alpha)}_{\rho(g)}$. Hence $\Sym(\ZZ_2^{m-1} \setminus \{\bar 0\}) \times \Sym(\widehat{\ZZ}_2)$ is a subgroup of $\cW(\Gamma_\ZZ(\cT_C))$. Since $L_{x_g}$ is a homogeneous automorphism of $\cT_C$ for $g\in\ZZ_2^{m-1}$, it follows from \eqref{product.cd.by.cartan} that $\cW := \Sym(\ZZ_2^{m-1}) \times \Sym(\widehat{\ZZ}_2)$ is a subgroup of $\cW(\Gamma_\ZZ(\cT_C))$. Without loss of generality, we can assume $m=3$ (for $m=1,2$, the proof is analogous).

Let $\varphi\in\cW(\Gamma_\ZZ(\cV_C))$. We claim that $\varphi\in\cW$. If we compose $\varphi$ with an element of $\cW$, we can assume that $\varphi$ fixes the subspace $(\FF v^{\charTrivial}_0)^+ = \FF e_1^+$, and therefore $\varphi$ also fixes the subspace with opposite degree, which is the subspace $(\FF v^{\omega}_0)^- = \FF e_2^-$. Then, the Peirce subspaces in Proposition~\ref{peirce.cartan} associated with the homogeneous idempotent $e^{\charTrivial}_0$ are fixed too. Consequently, $\varphi$ fixes the homogeneous component $(\FF v^\omega_0)^+ = \FF e_2^+$ and permutes the homogeneous components of $\bigoplus_{0\neq g\in\ZZ_2^{m-1}} (\FF v_g^{\charTrivial})^+ = \bigoplus_{i=1}^3 \FF u_i^+$. Then, composing $\varphi$ again with an element of $\cW$, we can assume that $\varphi$ fixes each of the subspaces $(\FF v_g^{\charTrivial})^+$ for $g\in\ZZ_2^{m-1}$. Again, since the Peirce subspaces are preserved, it follows that each of the subspaces $(\FF v_g^\omega)^+$ for $g\in\ZZ_2^{m-1}$ are fixed too, so that $\varphi$ is the identity map. We have proven that $\cW(\Gamma_\ZZ(\cV_C)) \leq \cW$. Consequently, we have that $\cW \leq \cW(\Gamma_\ZZ(\cT_C)) \leq \cW(\Gamma_\ZZ(\cV_C)) \leq \cW$, and the result follows.
\end{proof}

\begin{remark}
Let $C$ be a Hurwitz algebra. Consider the root decomposition of $\kan(\cV_C)$ associated with the Cartan grading $\Gamma_\ZZ(\cV_C)$ on $\cV_C$. Let $x_\alpha$, $x_\beta$ be elements of the Cartan basis of $\cV_C$ with associated roots $\alpha$, $\beta$. By Theorem~\ref{weyl.Cartan}, the action of the Weyl group on the homogenous components is transitive; thus all the roots related to the homogeneous components of $\cV_C^\sigma$ have the same length, so that $\left\lVert \alpha \right\rVert = \left\lVert \beta \right\rVert$. Hence, Proposition~\ref{peirce_constants} shows that the left Peirce constant relating $x_\alpha$ and $x_\beta$ is exactly $\cos(\alpha, \beta)$. Since the left Peirce constants appearing in Proposition~\ref{peirce.cartan} are exactly $1$, $1/2$, $0$ and $-1/2$, the corresponding angles appearing between their roots are $0^\circ$, $60^\circ$, $90^\circ$ and $120^\circ$, respectively.
\end{remark}

\section{Induced fine gradings via the Kantor construction} \label{section.induced.gradings}

In this section we give a summary of the fine gradings on Lie algebras obtained, using the Kantor construction, from the fine gradings on Kantor pairs of Hurwitz type. We will first determine the associated Kantor-Lie algebra for each case.

\bigskip

Recall from \cite{Al79} that the Kantor construction can be regarded as
$$\kan(\cA) = \cS^- \oplus \cA^- \oplus (T_{\cA} \oplus \mathrm{Der}(\cA)) \oplus \cA^+ \oplus \cS^+,$$
where $T_{\cA}= \lbrace T_{x}:= V_{x,1} \med x\in \cA \rbrace$.

Let $(C,\inv)$ be a Hurwitz algebra, and recall that in this case $\mathrm{Der}(C)=\mathrm{Inder}(C)$.
By \cite[Corollary 6]{Al79}, $\kan(C)$ is a simple Lie algebra.
Recall also that if the dimension of $C$ is 1, 2, 4 or 8, then the dimension of $\mathrm{Der}(C)$ is 0, 0, 3 or 14, respectively. Therefore the dimension of $\kan(C)$ is 3, 8, 21 or 52, respectively. Consequently, in the case that $\dim C$ equals 1, 2 or 8, the simple Lie algebra $\kan(C)$ must be isomorphic to $\mathfrak{a}_1 = \mathfrak{sl}_2$, $\mathfrak{a}_2 = \mathfrak{psl}_3$ or $\mathfrak{f}_4$, respectively.
On the other hand, for the case $\dim C = 4$ there are two simple Lie algebras of dimension 21 (of types $B_3$ and $C_3$), but since
the type of the main grading on $\kan(C)$ is $(0,0,2,2,0,0,1)$ and there is no grading on $\mathfrak{o}_7$ ($B_3$) with such type (whereas there is one for the Lie algebra of type $C_3$) (see \cite{EKmon} and references therein), we can conclude that $\kan(C)$ is isomorphic to $\mathfrak{sp}_3 = \mathfrak{c}_3$. 
Note that the Lie algebras we obtained are the ones appearing in the first row (or column) of the well-known Freudenthal magic square.

\smallskip

Recall from sections above that, for the Hurwitz Kantor pair $\cV_C$ in the case with $\chr\FF = 3$ and $\dim C = 2$, the results related to automorphisms and gradings are different. The reason of this is that $\kan(C) = \mathfrak{a}_2$ is exceptional in this case, and it is also well-known that $\Aut(\mathfrak{a}_2)$ is an exceptional Lie group of type $G_2$ \cite[\S 7.2, 7.3]{Ste61}. Note that in \cite{EKmon}, the classification of fine gradings on simple Lie algebras of type $A$ required a different treating in the case of $\mathfrak{a}_2$. Also, recall that $\mathfrak{a}_2 = \mathfrak{sl}_3$ if $\chr\FF \neq 3$, but $\mathfrak{a}_2 = \mathfrak{sl}_3 / Z(\mathfrak{sl}_3)$ with $Z(\mathfrak{sl}_3) = \FF I_3$ if $\chr\FF = 3$.

\begin{proposition}
Let $C$ be a Hurwitz algebra with $\dim C = 2^m > 1$, and either $\chr\FF \neq 3$ or $\dim C \neq 2$. The Cayley-Dickson grading on the Kantor pair $\cV_C$ extends to a fine grading on $\kan(\cV_C)$ with universal group $\ZZ \times \ZZ_2^m$. For each possible case, $\dim C = 2$, $4$, $8$, the type of the grading is $(8)$, $(15,3)$ and $(31,0,7)$, respectively.
\end{proposition}
\begin{proof}
The first part follows from Prop.~\ref{gradings.correspondence.from.pairs}. Consider the case $\dim C = 8$. Denote $\cV = \cV_\cC$. Recall that both subspaces $\kan(\cV)^1 = \cV^+$ and $\kan(\cV)^{-1} = \cV^-$ consist of eight $1$-dimensional homogeneous components. Also, note that $\dim \kan(\cV)^0 = 22$ and $\dim \kan(\cV)^{\pm 2} = 7$.

For $x_g$, $x_h$ in the Cayley-Dickson basis $\cB_{\CD}(\cC)$ of $\cC$, we have that $K(x_g^+, x_h^+)= L_{\psi(x_g, x_h)}$ where $\psi(x,y) := x\bar y - y\bar x$. Therefore, $K(x_g^+, x_h^+) = 0$ if $g = h$, that is, $\kan(\cV)^2_{(2, e)} = 0$, where $e$ is the neutral element of $\ZZ_2^3$. Hence, $\supp\kan(\cV)^2 = \{(2,g) \med e\neq g\in\ZZ_2^3 \}$. The Weyl group in Theorem~\ref{weyl.CD} shows that the homogeneous components of $\supp\kan(\cV)^2$ are in the same orbit under the action by automorphisms, so that they have the same dimension, and consequently they are $1$-dimensional. The same arguments hold for $\supp\kan(\cV)^{-2}$.

Note that for $x_g \in\cB_{\CD}(\cC)$ we have that $D^{\sigma}(x_g, x_g)= \id_{\cV^{\sigma}}$ (that is, $x_g^-$ is the conjugate inverse of $x_g^+$), so that $\dim \kan(\cV)^0_{(0,e)} = 1$. Again, the Weyl group in Theorem~\ref{weyl.CD} shows that the homogeneous components with degrees $\{(0, g) \med e\neq g\in\ZZ_2^3\}$ are in the same orbit under the action by automorphisms, so that they must have the same dimension, which must be 3. It follows that the type of the grading on $\kan(\cV)$ is $(31,0,7)$ (note that this coincides with the type of the grading on $\mathfrak{f}_4$ given in \cite[Corollary 5.40]{EKmon}). The proof is analogous in the cases $\dim C = 2$, $4$.
\end{proof}

\begin{proposition}
Let $C$ be a Hurwitz algebra with $\dim C = 2^m$. The Cartan grading on the Kantor pair $\cV_C$ extends to a fine grading on $\kan(\cV_C)$ with universal group $\ZZ^{m+1}$, that is, the Cartan grading on $\kan(\cV_C)$.
\end{proposition}
\begin{proof}
Consequence of the correspondence given in Prop.~\ref{gradings.correspondence.from.pairs} and the fact that Cartan gradings on Lie algebras are produced by maximal tori.
\end{proof}

In Figure~\ref{tabla.Kantor} we summarize the general information for each fine Kantor-compatible grading on the Kantor-Lie algebras of our study, including the type and universal group. (The types of the gradings follow from the classification of fine gradings on the classical simple Lie algebras \cite{EKmon}.) Recall that we denote Cartan gradings and Cayley-Dickson gradings by $\Gamma_\ZZ$ and $\Gamma_\CD$, respectively.

\begin{center}
\begin{figure}[htbp]
\begin{tabular}{|c||c|c|c|c|}
  \hline
	$\dim C$ & 1 & 2 & 4 & 8 \\
	$\kan(C)$ & $\mathfrak{a}_1$ & $\mathfrak{a}_2$ & $\mathfrak{c}_3$ & $\mathfrak{f}_4$ \\
	\hline
	$\Univ(\Gamma_\ZZ)$ & $\ZZ$ & $\ZZ^2$ & $\ZZ^3$ & $\ZZ^4$ \\
	Type($\Gamma_\ZZ$) & $(3)$ & $(6, 1)$ & $(18, 0, 1)$ & $(48, 0, 0, 1)$ \\
	\hline
	$\Univ(\Gamma_\CD)$ & $-$ & $\ZZ\times\ZZ_2 \ (\chr\FF \neq 3)$ & $\ZZ\times\ZZ_2^2$ & $\ZZ\times\ZZ_2^3$ \\
	Type($\Gamma_\CD$) & $-$ & $(8)$ & $(15,3)$ & $(31,0,7)$ \\	
	\hline
\end{tabular}
\caption{Fine gradings obtained via the Kantor construction from Kantor pairs of Hurwitz type}
\label{tabla.Kantor}
\end{figure}
\end{center}

\bigskip

\textbf{Acknowledgements}
The authors are very thankful to Alberto Elduque for providing important references and some valuable advice. Thanks are also due to the anonymous referee, for reading the manuscript and making some helpful comments.



\begin{thebibliography}{99}
\bibliographystyle{alpha}

\bibitem[Al78]{Al78}
B.N.~Allison, \textit{A class of nonassociative algebras with involution containing the class of Jordan algebras}, Math. Ann. \textbf{237} (1978), 133--156.

\bibitem[Al79]{Al79} B.~Allison, \textit{Models of isotropic simple Lie algebras},
 Communications in Algebra, 7:17, 1835-1875 (1979).

\bibitem[Ara17]{Ara17}
D. Aranda-Orna, \emph{Fine gradings on simple exceptional Jordan pairs and triple systems}, J. Algebra \textbf{491} (2017) 517--572.

\bibitem[AF84]{AF84}
B.~Allison and J.~Faulkner, \textit{A Cayley--Dickson process for a class of structurable algebras}, Trans. Amer. Math. Soc. \textbf{283} (1984), no. 1,  185--210.

\bibitem[AF92]{AF92}
B.N.~Allison and J.R.~Faulkner, \textit{Norms on structurable algebras}, Comm. Algebra \textbf{20} (1992), no.~1, 155--188.

\bibitem[AF99]{AF99}
B.~Allison and J.~Faulkner, \textit{Elementary groups and invertibility for Kantor pairs}, Comm. Algebra \textbf{27} (1999), 519--556.

\bibitem[AFS17]{AFS17}
B.~Allison, J.~Faulkner and O.~Smirnov, \emph{Weyl images of Kantor pairs}, Canadian Journal of Mathematics \textbf{69} (2017) 721--766.

\bibitem[AM99]{AM99} H.~Albuquerque and S.~Majid, \textit{Quasialgebra structure of the octonions}, J. Algebra \textbf{220} (1999), no.~1, 188--224.

\bibitem[DET21]{DET21} A.~Daza-García, A.~Elduque and L.~Tang, \textit{Cross products, automorphisms, and gradings}, Linear Algebra and its Applications,
Vol. \textbf{610} (2021), 227--256. (To appear.) Preprint: \href{https://arxiv.org/abs/2006.10324}{arXiv:2006.10324}

\bibitem[Eld00]{Eld00}
A. Elduque, \emph{On triality and automorphisms and derivations of composition algebras}, Linear Algebra Appl. \textbf{314} (2000), no. 1--3, 49--74.

\bibitem[Eld96]{Eld96}
A. Elduque, \emph{On a class of ternary composition algebras}, J.~Korean Math. Soc. \textbf{33} (1996), no. 1, 183--203.

\bibitem[Eld98]{Eld98}
A. Elduque, \emph{Gradings on octonions}, J. Algebra \textbf{207} (1998), no. 1, 342--354.

\bibitem[EK13]{EKmon} A. Elduque and M. Kochetov, \textit{Gradings on simple {L}ie algebras}, Mathematical Surveys and Monographs \textbf{189},  American Mathematical Society, Providence, RI, 2013.

\bibitem[F94]{F94}
J.R.~Faulkner, \emph{Structurable triples, Lie triples, and symmetric spaces}, Forum Math. \textbf{6} (1994), no. 5, 637--650. 

\bibitem[H78]{H78} J. E. Humphreys, \emph{Introduction to Lie algebras and representation theory}, Second printing, revised. Graduate Texts in Mathematics, 9. Springer-Verlag, New York-Berlin, 1978.

\bibitem[J89]{J89}
N. Jacobson, \emph{Basic algebra II}, Second Edition, W. H. Freeman and Company, New York, (1989).

\bibitem[K72]{K72}
I.L.~Kantor, \emph{Certain generalizations of Jordan algebras} (Russian), Trudy Sem. Vektor. Tenzor.
Anal. \textbf{16} (1972), 407--499.

\bibitem[K73]{K73}
I.L.~Kantor, \emph{Models of the exceptional Lie algebras}, Soviet Math. Dokl. \textbf{14} (1973), 254--258.

\bibitem[KK03]{KK03}
I. Kantor, N. Kamiya, \textit{A Peirce Decomposition for Generalized Jordan Triple Systems of Second Order}, Comm. Algebra \textbf{31} (2003), no. 12, 5875--5913.

\bibitem[L75]{L75}
O. Loos, {\em Jordan Pairs}, Lecture Notes in Mathematics, Vol. \textbf{460}. Springer-Verlag, Berlin-New York, 1975.

\bibitem[MC03]{MC03}
K. McCrimmon, \textit{A Taste of Jordan Algebras} (2003), Springer-Verlag.

\bibitem[Smi92]{Smi92} O.N.~Smirnov, \textit{Simple and semisimple structurable algebras}, Proceedings of the International Conference on Algebra, Part 2 (Novosibirsk, 1989), Contemp. Math. \textbf{131}, Part~2, Amer. Math. Soc., Providence, RI, 1992, pp. 685--694.

\bibitem[Ste61]{Ste61}
R. Steinberg, \textit{Automorphisms of classical Lie algebras}, Pacific J. Math. \textbf{11} (1961), 1119--1129.

\bibitem[ZSSS82]{ZSSS82}
K.A. Zhevlakov, A.M. Slin'ko, I.P. Shestakov, A.I. Shirshov, {\it Rings that are nearly associative}, Academic Press, 1982.

\end{thebibliography}
\end{document}